\documentclass [10pt,a4paper,reqno]{amsart}

\usepackage[utf8]{inputenc}
\usepackage{amsmath}
\usepackage{lmodern}
\usepackage{amssymb}
\usepackage{xypic}
\usepackage{tikz-cd}

\tikzcdset{scale cd/.style={every label/.append style={scale=#1},
    cells={nodes={scale=#1}}}}
\usetikzlibrary{arrows,decorations.pathmorphing,backgrounds,positioning,fit,petri,calc,shapes.misc,decorations.markings}
\tikzset{degil/.style={
        decoration={markings,
            mark= at position 0.5 with {
                \node[transform shape] (tempnode) {$\backslash$};
            }
        },
        postaction={decorate}
    }
}

\tikzset{degil3/.style={
        decoration={markings,
            mark= at position 0.3 with {
                \node[transform shape] (tempnode) {$\backslash$};
            }
        },
        postaction={decorate}
    }
}

\tikzset{degil4/.style={
        decoration={markings,
            mark= at position 0.4 with {
                \node[transform shape] (tempnode) {$\backslash$};
            }
        },
        postaction={decorate}
    }
}

\xyoption{all}

\newtheorem{thm}{Theorem}[section] 

\newtheorem{conj}[thm]{Conjecture}
\newtheorem{cor}[thm]{Corollary}
\newtheorem{defn}[thm]{Definition}

\newtheorem{lem}[thm]{Lemma}
\newtheorem{prop}[thm]{Proposition}

\newtheorem{rem}[thm]{Remark}

\newtheorem{ques}[thm]{Question}

\def\Span{{\operatorname{Span}}}
\def\Ker{{\operatorname{Ker}}}
\def\i{{\operatorname{in}}}

\begin{document}

\title[The space of growth functions]{Gaps and approximations in the space of growth functions}



\thanks{The author thanks Jason Bell and Efim Zelmanov for related inspiring discussions.}

\begin{abstract}

An important problem in combinatorial noncommutative algebra is to characterize the growth functions of finitely generated algebras (equivalently, semigroups, or hereditary languages).

The growth function of every finitely generated, infinite-dimensional algebra is increasing and submultiplicative.
The question of to what extent these natural necessary conditions are also sufficient --- and in particular, whether they are sufficient at least for sufficiently rapid functions --- was posed and studied by various authors and has attracted a flurry of research.

While every increasing and submultiplicative function is realizable as a growth function up to a linear error term, we show that there exist
arbitrarily rapid increasing submultiplicative functions which are not equivalent to the growth of any algebra, thus resolving the aforementioned problem and settling a question posed by Zelmanov (and repeated by Alahmadi-Alsulami-Jain-Zelmanov). These can be interpreted as `holes' in the space of growth functions, accumulating to exponential functions in the order topology.

We show that there exist monomial algebras and hereditary languages whose growth functions encode the existence of non-prolongable words, and algebras whose growth functions encode the existence of nilpotent ideals (in the graded case). This negatively solves another conjecture of Alahmadi-Alsulami-Jain-Zelmanov in the graded case.

\end{abstract}

\author{Be'eri Greenfeld}

\address{Department of Mathematics, University of California,
San Diego,
La Jolla, CA, 92093, USA} %
\email{bgreenfeld@ucsd.edu}


\maketitle

\addtocontents{toc}{\protect\setcounter{tocdepth}{1}}





\section{Introduction}

\subsection{Growth functions} The question of `how do algebras grow?', or, which functions can be realized as growth functions of algebras, is a major problem in the junction of several mathematical fields including noncommutative algebra, combinatorics of infinite words, symbolic dynamics, formal languages, self-similarity and more. It is considered a basic open problems in combinatorial noncommutative algebra (e.g. \cite{BellZelmanov, Smoktunowicz, SmoktunowiczBartholdi}).

Consider a finitely generated associative algebra $A$ over an arbitrary field $F$. Suppose that $A$ is infinite-dimensional as an $F$-vector space.
Fixing a finite-dimensional generating subspace $A=F\left<V\right>$, the growth of $A$ with respect to $V$ is defined to be the function:
$$\gamma_{A,V}(n)=\dim_F \left(F+V+V^2+\cdots+V^n\right)$$
If $1\in V$ then equivalently $\gamma_{A,V}(n)=\dim_F V^n$. This function evidently depends on the choice of $V$, but might change only up to the following equivalence relation.
We say that $f\preceq g$ if $f(n)\leq g(Cn)$ for some $C>0$ and for all $n\in \mathbb{N}$
, and $f\sim g$ (asymptotically equivalent) if $f\preceq g$ and $g\preceq f$. Therefore when talking about `the growth of an algebra' one refers to $\gamma_A(n)$ as the equivalence class of the functions $\gamma_{A,V}(n)$ (for some $V$) under the equivalence relation $\sim$. 

The class of growth functions of finitely generated algebras is not only an algebraic entity: it coincides with the classes of growth functions of finitely generated semigroups and hereditary languages, and (strictly) contains the class of growth functions of finitely generated groups. 


There are obvious properties necessarily satisfied by growth functions of algebras; such functions are always:
\begin{itemize}
    \item \textit{Increasing} (namely, $f(n)<f(n+1)$); and
    \item \textit{Submultiplicative} (namely, $f(n+m)\leq f(n)f(m)$).
\end{itemize}  
The main goal in studying the variety of possible growth functions is to investigate to what extent these conditions are in fact sufficient. These properties are referred to as the obvious necessary conditions for growth functions \cite{BartholdiErschler,BellZelmanov,GrigorchukPak,SmoktunowiczBartholdi}.

\subsection{The space of growth functions} The first gap result on impossibility of growth functions (apart from the two properties mentioned above --- increasing and multiplicative) is Bergman's gap theorem \cite{BergmanGap}, which asserts that a super-linear growth function must be at least quadratic; for infinite words, this was discovered by Morse and Hedlund \cite{MorseHedlund}. In fact, the discrete derivative $f'(n)=~f(n)-f(n-1)$ is either eventually constant or grows at least linearly. 
As mentioned in \cite{BartholdiErschler}, Bergman's gap is the only gap in the space of growth functions known so far.

Several attempts have been made to realize as wide as possible variety of increasing and submultiplicative functions as growth functions of associative algebras.
The first major step towards this was achieved by Smoktunowicz and Bartholdi \cite[Theorem~C]{SmoktunowiczBartholdi}, who proved that every increasing and submultiplicative function is equivalent to the growth of an associative algebra, up to a polynomial factor. 
Namely, if $f\colon\mathbb{N}\rightarrow \mathbb{N}$ is a submultiplicative and increasing function, then there exists a finitely generated monomial algebra $B$ whose growth function satisfies: $$f(n)\preceq \gamma_B(n)\preceq n^3f(n).$$
In particular, if $f$ is an increasing, submultiplicative functions such that $f(n)\sim~nf(n)$, then $f$ is equivalent to the growth function of an algebra. This was interpreted in \cite{BartholdiErschler,SmoktunowiczBartholdi} as follows: any `sufficiently regular' function faster than $n^{\ln n}$ is equivalent to the growth function of an algebra.

To compare, the class of growth functions of \textit{groups} is much more restricted. Bartholdi and Erschler \cite{BartholdiErschler} realized a broad variety of intermediate functions as growth rates of finitely generated groups; they proved that any function $f\colon\mathbb{N}\rightarrow \mathbb{N}$ which grows `uniformly faster' than $\exp(n^\alpha)$ is equivalent to a growth function of some finitely generated group\footnote{Here $\alpha=\log 2/\log \eta\approx 0.7674$, where $\eta$ is the positive root of $X^3-X^2-2X-4$. `Uniformly faster' here means that $f(2n)\leq f(n)^2\leq f(\eta n)$ for $n\gg 1$.}.
Kassabov and Pak \cite{KassabovPak} constructed groups with oscillating growth functions, exhibiting surprising flexibility within the space of growth functions of groups. The possibility of a parallel pathological behavior for growth of semigroups and algebras was already known by interesting constructions of Trofimov \cite{Semigroup} and at an extreme by Belov-Borisenko-Latyshev \cite{BBL1997}, and was recently demonstrated in the setting of Lie algebras by Petrogradsky \cite{PetrogradskyOscillating}.
The question of providing a complete characterization of growth functions of groups is widely open. A gap between polynomial growth functions and a certain quasi-polynomial function is known \cite{ShalomTao}, and Grigorchuk posed the tantalizing conjecture that this gap is actually wider, so there are no super-polynomial growth functions of groups which are slower than $\exp(\sqrt{n})$ (see \cite{GriGap}). Grigorchuk proved the latter gap for residually-nilpotent groups \cite{GriResp}, using growth analysis of graded Lie algebras, but the general case is open.


The space of growth functions of semigroups and algebras is much richer than the space of growth functions of groups. To what extent are the obvious necessary conditions of being increasing and submultiplicative also sufficient for a (super-quadratic, to avoid Bergman's gap) function to be realizable as the growth of an algebra? Bell and Zelmanov \cite{BellZelmanov} identified an additional condition (on discrete derivatives) satisfied by all growth functions; their remarkable result is that in fact, \textit{every increasing function satisfying this condition is equivalent to the growth of an associative algebra}. They proved:
\begin{thm}[{\cite[Theorem~1.1]{BellZelmanov}}] 
A growth function of an algebra is asymptotically equivalent to a constant function, a linear function, or a weakly increasing function $f\colon\mathbb{N}\rightarrow\mathbb{N}$ with the following properties:
\begin{enumerate}
    \item $f'(n)\geq n+1$ for all $n\in \mathbb{N}$;
    \item $f'(m)\leq f'(n)^2$ for all $m\in \{n,\dots,2n\}$.
\end{enumerate}
Conversely, if $f(n)$ is either a constant function, a linear function, or a function
with the above properties then it is asymptotically equivalent to the growth function of a finitely generated algebra.
\end{thm}

An interesting byproduct of this result, observed in Proposition \ref{linear error term}, is that every increasing submultiplicative function is realizable as the growth function of an algebra up to a \textit{linear} error term, thereby improving \cite[Theorem~C]{SmoktunowiczBartholdi}. This is the best possible approximation, in view of Bergman's gap theorem.

As the authors suggest, one can interpret this theorem as saying that other than the necessary condition that $f'(m)\leq f'(n)^2$ for all $m\in \{n,\dots,2n\}$, which is related to submultiplicativity, the only additional constraints required for being realizable as the growth of an algebra are those arising from Bergman's gap theorem and the elementary gap that an algebra of sublinear growth is finite-dimensional and hence its growth functions is eventually constant. 

However, it seems that there is no natural characterization of whether a given function is equivalent to a function satisfying the above condition on discrete derivatives: an increasing, submultiplicative function might be equivalent to a growth function of an algebra, yet need not satisfy condition (2) from \cite[Theorem~1.1]{BellZelmanov}, which is indeed not preserved under asymptotic equivalence.

Therefore, we remain with the following fundamental question, which was posed on various stages:

\begin{ques}[{Alahmadi-Alsulami-Jain-Zelmlanov, 2017 \cite{AAJZ_announcement}; Zelmanov, 2017 \cite{Spa}}] \label{Zel_que}
Let $f\colon\mathbb{N}\rightarrow \mathbb{N}$ be an increasing submultiplicative function such that $f(n)\succeq n^2$. Is $f$ asymptotically equivalent to the growth of some finitely generated associative algebra?
\end{ques}

It seems that this question had appeared in various other articles in slightly modified versions (e.g. \cite[p.~423]{SmoktunowiczBartholdi}). It is reasonable to further weaken the question and ask whether only `sufficiently rapid' increasing and submultiplicative functions (i.e. functions which are faster than a given subexponential function) are realizable \cite{Zelmanov private communication}.

Our first main result answers Question \ref{Zel_que} in its full generality, refuting even its weaker version. Namely, we prove:
\begin{thm}\label{main thm submul}
Let $g\colon\mathbb{N}\rightarrow\mathbb{N}$ be a subexponential function. Then there exists a function $f\colon\mathbb{N}\rightarrow \mathbb{N}$ such that:
\begin{itemize}
    \item $f$ is submultiplicative and increasing;
    \item $f\succeq g$;
    \item $f$ is not equivalent to the growth function of any finitely generated algebra.
\end{itemize}
\end{thm}
Moreover, $f'(n)\geq n+1$, to emphasize the independence with Bergman's gap theorem\footnote{Even though it is a-priori possible that any sufficinelty rapid, increasing and submultiplicative function is equivalent to a function which additionally satisfies $f'(n)\geq n+1$.}. Notice that a submultiplicative exponentially growing function is equivalent to the growth of a free algebra, so our functions are a fortiori subexponential.

Theorem \ref{main thm submul} exhibits the existence of new `holes' in the space of growth functions. 
This phenomenon proves the necessity of the error terms in \cite[Theorem~C]{SmoktunowiczBartholdi} and Proposition \ref{linear error term}, and hints that there is in fact no characterization in an asymptotic language of growth functions within the class of `natural candidates', namely, increasing and submultiplicative functions. This justifies and emphasizes the importance of using additional characteristics of functions in the attempt to characterize growth functions, such as a submultiplicativity-type condition on discrete derivatives as done in \cite[Theorem~1.1]{BellZelmanov}.

Theorem \ref{main thm submul} has the following topological interpretation. Let $\mathcal{F}$ be the set of $\sim$-equivalence classes of functions $f\colon \mathbb{N}\rightarrow \mathbb{N}$ which are increasing and submultiplicative. Then $\mathcal{F}$ is partially ordered by $\preceq$, and has a top point ($\exp(n)$) and a bottom point ($O(n)$). This partial order induces a topology (where a basis of open sets is given by open intervals).
The main problem now translates to understanding `how large' is the subsapce $\mathcal{F}\subseteq \mathcal{G}$ consisting of equivalence classes of growth functions of algebras.
So, while the top point $\exp(n)\in \mathcal{G}$, Theorem \ref{main thm submul} tells us that it lies in the closure of $\mathcal{G}\setminus \mathcal{F}$.

\subsection{Realization within restricted classes}

Let $f\colon\mathbb{N}\rightarrow\mathbb{N}$ be (equivalent to) a growth function of an algebra. Can we realize $f$ as the growth of an algebra with a deeper algebraic/geometric/dynamical structure rather than a pathological combinatorial construction? 
It was observed by the authors of \cite{BellZelmanov} that their construction (Theorem~1.1 therein) yields monomial algebras with many `singular' monomials, namely, monomials whose product with any other monomial is zero. In fact, the prime radical of these algebras are so large that the quotients with respect to them are a polynomial ring in one variable.

It is therefore natural to ask whether one can realize growth functions of algebras within the class of monomial algebras whose language of non-zero monomials is prolongable, namely, every non-zero monomial can be extended to a longer non-zero monomial; this direction was proposed to us by Zelmanov. The following construction shows that this is impossible in general. The growth of a hereditary formal language is given by counting the number of its words of length at most $n$.

\begin{thm} \label{Prolongable}
There exists a hereditary language whose growth is not equivalent to the growth of any prolongable hereditary language.
\end{thm}

Equivalently, there exists a monomial algebra whose growth function is not equivalent to the growth of any monomial algebra without singular monomials.
Hence the phenomenon that the monomial algebras constructed by Bell and Zelmanov \cite{BellZelmanov} contain many singular monomials is, in fact, inevitable.

Therefore, one may ask which growth functions of algebras are realizable within algebras having a deeper algebraic structure, or a more natural mathematical origin rather than a pathological construction. By analogy from geometric group theory, while groups of intermediate growth were known for decades, only a few years ago a simple group of intermediate growth was constructed \cite{Nekrashevych simple}; simple algebras (and simple Lie algebras) of prescribed intermediate growth were constructed by the author in \cite{BGProc,BGJalg}. Within the class of monomial algebras, important algebraic properties (prolongability, primeness, just-infiniteness) reflect dynamical properties of an underlying subshift from which the algebra originates.

A ring is prime if the product of non-zero ideals is non-zero; this is a non-commutative extension of being an integral domain. In particular, prime rings contain no non-zero nilpotent ideals. A ring is primitive if it acts faithfully on an irreducible module. This is an important representation-theoretic notion, which implies primeness. For a comprehensive survey of primitive rings, see \cite{BellPrimitive}.

\begin{conj}[{Alahmadi-Alsulami-Jain-Zelmlanov, 2017 \cite{AAJZ_announcement}; Zelmanov, 2017 \cite{Spa}}] \label{ZelmanovsConjecture}
The following are equivalent for a function $f\colon\mathbb{N}~\rightarrow~\mathbb{N}$:
\begin{itemize}
    \item $f$ is equivalent to the growth of a finitely generated algebra;
    \item $f$ is equivalent to the growth of a finitely generated primitive algebra;
    \item $f$ is equivalent to the growth of a finitely generated nil algebra.
\end{itemize}
\end{conj}
(Excluding linear functions, since by \cite{SmallStaffordWarfield} algebras of linear growth satisfy a polynomial identity and are therefore neither primitive nor nil \cite{Kaplansky}; an algebra is nil if all of its elements are nilpotent.)
Let us concentrate in primitive algebras. The above conjecture has a partial positive solution up to a quadratic error term:
\begin{prop} \label{primitive_proposition}
Let $f\colon \mathbb{N}\rightarrow \mathbb{N}$ be an increasing, submultiplicative function. Then there exists a finitely generated primitive algebra $A$ such that $f(n)\preceq \gamma_A(n)\preceq ~ n^2f(n)$.
\end{prop}
(This was done up to a weaker error term by the author in \cite[Section~3]{BGIsr}).
Surprisingly, there is also a partial \textit{negative} answer, if we restrict to graded algebras\footnote{Recall that the growth of any algebra is equivalent to the growth of a graded algebra; it is possible, however, that a function is equivalent to the growth of a primitive algebra, but not equivalent to the growth of any primitive graded algebra.}.


\begin{thm} \label{Graded}
There exists a finitely generated algebra $A$ such that:
\begin{enumerate}
    \item $\gamma_A(n)$ is not equivalent to the growth of any graded algebra without non-zero nilpotent ideals. In particular, it is not equivalent to the growth of any prime graded algebra.
    \item $\gamma_A(n)$ is not equivalent to the growth of any graded algebra with a regular homomgeneous element.
\end{enumerate}
\end{thm}

In particular, there exist infinite-dimensional monomial algebras all of whose graded deformations contain non-zero nilpotent ideals.
If $A$ is a graded algebra then $A[t]$ obviously has a regular homogeneous element, so up to a linear error term, all growth functions of algebras are equivalent to growth functions of graded algebras with regular homogeneous elements.
It should be mentioned that there exist graded prime algebras without regular homogeneous elements, and graded algebras with regular homogeneous elements which are not prime.



We start with an introductory section (Section \ref{sec:roadmap}), illustrating the hierarchy of the main classes of growth functions of algebras, languages and subshifts, followed by a short section (Section \ref{sec:realizing}) presenting positive realization results which emphasize the significance of our main results.

\section{A roadmap of growth: algebra, combinatorics and dynamics} \label{sec:roadmap}

\subsection{Algebras} \label{subsec:algebras} All algebras in this paper are associative and finitely generated over an arbitrary base field $F$. Let $A$ be an $F$-algebra, generated by a finite-dimensional subspace $V$. We say that $A$ is \textit{graded} if it decomposes into a direct sum of homogeneous components $A=\bigoplus_{n=0}^{\infty} A_n$ such that $A_nA_m\subseteq A_{n+m}$. Throughout the paper, graded algebras are assumed to have finite-dimensional homogeneous components.
An algebra is \textit{monomial} if it can be presented as a quotient of a free algebra modulo monomial relations; a monomial algebra is graded by any assignment of (positive integer) degrees to its generators in the monomial presentation.
Monomial algebras are prime if and only if for any two non-zero monomials $u,v$ there exists a monomial $w$ such that $uwv$ is non-zero. For more on monomial algebras, see \cite{BBL1997}.

Let $A=F\left<x_1,\dots,x_d\right>/I$ be an arbitrary algebra. Consider the ideal $\i(I)\triangleleft F\left<x_1,\dots,x_d\right>$ generated by the initial monomials of elements in $I$, with respect to the order $x_1<\cdots<x_d$. Then the growth of $A$ is equal to the growth of $F\left<x_1,\dots,x_d\right>/\i(I)$, the latter being a monomial (hence graded) algebra. Therefore, while the class of graded algebras is very restricted compared to the whole class of algebras, the classes of growth functions of algebras, graded algebras and monomial algebras coincide. We remark that $F\left<x_1,\dots,x_d\right>/\i(I)$ might fail to be prime even when $F\left<x_1,\dots,x_d\right>/I$ is. The classical comprehensive reference for growth of algebras is \cite{KrauseLenagan}.

\subsection{Languages and subshifts} \label{subsec:languages and subshifts}

Let $\Sigma=\{x_1,\dots,x_d\}$ be a finite alphabet. A language over $\Sigma$ is a subset $\mathcal{L}\subseteq \Sigma^*=\bigcup_{i=1}^{\infty} \Sigma^i$, and its elements are called words.
Following \cite{BellZelmanov}, the growth function of $\mathcal{L}$ is $\gamma_{\mathcal{L}}(n)=\#\{\text{length $\leq n$ words in $\mathcal{L}$}\}$.
A language is \textit{hereditary} if it is closed under taking subwords. The set of non-zero monomials in a finitely generated algebra forms a hereditary language, and conversely, any hereditary language gives rise to a monomial algebra spanned by its words: $F\left<x_1,\dots,x_d\right>/\left<w\ \text{monomial}\ |\ w\notin \mathcal{L}\right>$.
Hence the class of growth functions of algebras coincides with the class of growth functions of hereditary languages.
We say that a language $\mathcal{L}$ is \textit{prolongable} if for every $w\in \mathcal{L}$ there exist $u,v$ such that $uwv\in\mathcal{L}$.

\subsubsection{Subshifts}

The set $\Sigma^{\mathbb{Z}}$, endowed with the product topology, is a dynamical system with respect to the shift operator $T$. A subshift is a closed, shift-invariant subset $X\subseteq \Sigma^{\mathbb{Z}}$.

With any subshift $X\subseteq \Sigma^{\mathbb{Z}}$ a hereditary language $\mathcal{L}(X)$ over $\Sigma$ is associated, namely, the language of its finite factors. 
The \textit{factor complexity} of $X$ is the function $p_X(n)=\#\{\text{length-$n$ factors of $X$}\}$. To align with the framework of hereditary languages, let us define the \textit{growth function} of $X$ to be $P_X(n)=\#\{\text{length $\leq n$ factors of } X\}$; thus $P_X(n)=\gamma_{\mathcal{L}(X)}(n)$.
We call a subshift $X$ \textit{transitive} if it has a dense orbit\footnote{In the more general setting of dynamical systems, transitivity may refer to having any non-empty open subset intersect any other non-empty open subset, after a suitable shift.}: $X=\overline{\{T^iw\}}_{i\in\mathbb{Z}}$ for some (bi-)infinite word $w\in \Sigma^{\mathbb{Z}}$. In this case we refer to the complexity (resp.~growth) function of the subshift as of the bi-infinite word $w$, denoted $p_X(n)=p_w(n)$ (resp.~$P_X(n)=P_w(n)$). For a survey on complexity functions of infinite words, see \cite{complexity_infinite_words}.

A subshift $X$ is called a \textit{Cantor set} if it is homeomorphic to the Cantor set, which means that $X$ has no isolated
points.

An infinite word in which every factor appears infinitely many times is called \textit{recurrent}, and a subshift which has a dense orbit of a recurrent word is \textit{irreducible}: alternatively, for any two factors $u,v\in \mathcal{L}(X)$ there exists $w$ such that $uwv\in\mathcal{L}(X)$.
A transitive aperiodic subshift has no isolated points if and only if it is irreducible.

\subsubsection{Groupoids and convolution algebras}
Corresponding to any (aperiodic) 
subshift $X\subseteq \Sigma^{\mathbb{Z}}$ one can endow the topological space $\mathbb{Z}\times X$ with a (topological) groupoid structure, denoted $\mathfrak{G}_X$, where partial multiplication is given by:
$$ (n,T^m u)\cdot (m,u)=(n+m,T^{n+m} u) $$
Units are elements of the form $gg^{-1}$, so in the above case units are of the form $(0,x)$ for $x\in X$.
The \textit{convolution algebra} over a (discrete) field $F$ of $\mathfrak{G}_X$ is the ring of compactly supported continuous functions $\varphi\colon\mathfrak{G}_X\rightarrow F$ equipped with the pointwise addition and convolution operation. For more on convolution algebras of groupoids associated with subshifts, see \cite{Nekrashevych}.

\subsubsection{Equivalence of growth functions of languages and subshifts}
The equivalence relation of growth functions $f\sim g$, which is natural in the context of geometric group theory and growth of algebras (where one is not interested in specifying a generating set) has the following interpretation for languages, and hence for subshifts. 

Let $\mathcal{L}$ be a subshift over an alphabet $\Sigma=~\{x_1,\dots,x_d\}$. To each letter one can assign a `weight': $w(x_1),\dots,w(x_d)\in~\mathbb{R}_{>0}$. Accordingly, one can define the weighted growth, measuring the number of words of total weight $\leq n$ in $\mathcal{L}$. Then the various weighted complexity functions $P_{\mathcal{L},w}(n)$ resulting from any such weighting might be distinct, but are always equivalent: $\gamma_{\mathcal{L},w}(n)\sim \gamma_{\mathcal{L}}(n)$.
Similarly, for a subshift $X\subseteq \Sigma^{\mathbb{Z}}$, all weighted growth functions are equivalent to its standard growth function.


\subsubsection{Subshifts correspond to hereditary prolongable languages} \label{subsubsec:subshifts correspond to hereditary prolongable languages}

The set of all factors of a given subshift forms a hereditary prolongable language; conversely, let $\mathcal{L}$ be a hereditary prolongable language over an alphabet $\Sigma=\{x_1,\dots,x_d\}$. Enumerate its words: $w_1,w_2,\dots$ and notice that since $\mathcal{L}$ is prolongable we can find bi-infinite words $W_1,W_2,\dots$ whose factors belong to $\mathcal{L}$ and $w_i$ factors $W_i$ for each $i\in\mathbb{N}$. Then the closure of the shift-orbits of $\{W_1,W_2,\dots\}$ in $\Sigma^{\mathbb{Z}}$ forms a subshift whose set of factors is exactly $\mathcal{L}$.
Therefore, the class of growth functions of subshifts coincides with the class of growth functions of hereditary prolongable languages.

\subsection{Discrete derivatives}

Let $f\colon\mathbb{N}\rightarrow\mathbb{N}$ be a function. Its \textit{discrete derivative} is the function $f'\colon\mathbb{N}\rightarrow\mathbb{N}$ given by $f'(n)=f(n)-f(n-1)$ (we set $f'(1)=f(1)$).

\subsubsection{Graded algebras} \label{subsubsec:graded algebras} Let $A$ be an algebra. If $A=\bigoplus_{n=0}^{\infty} A_n$ is (connected) graded, generated in degree $1$ then the growth function corresponding to the generating subspace $F+A_1$ is $\gamma_A(n)=\dim_F \bigoplus_{i=0}^{n} A_i$ and $\gamma_A'(n)=\dim_F A_n$.

Suppose that $A$ has a regular element (namely, a non zero-divisor) which is homogeneous, say, $a\in A_k$. Notice that since $A$ is generated in degree $1$ then $A$ is a finitely generated module over the Veronese subring $\bigoplus_{i=0}^{\infty} A_{kn}$, which is again generated in degree $1$ with respect to a modified grading ($\widetilde{\deg}(r)=\frac{1}{k}\deg(r)$). Since an algebra shares its growth with any subalgebra over which it is a finitely generated module, we may therefore assume that $a$ has degree $1$.
Then $\dim_F A_n=\dim_F aA_n\leq \dim_F A_{n+1}$, so $\gamma_A'(n)$ is non-decreasing, or equivalently $\gamma_A''(n)\geq 0$. Notice that having a non-decreasing discrete derivative is not invariant under equivalence, namely, it is possible that growth functions of the same algebra with respect to distinct generating subspaces will not share this property.

As observed by Bell and Zelmanov \cite{BellZelmanov}, the growth function of any finitely generated algebra with respect to any generating subspace has a submultiplicative discrete derivative. However, a function might be equivalent to a growth function of an algebra, yet have a non-submultiplicative discrete derivative.

\subsubsection{Languages} Let $\mathcal{L}$ be a hereditary language over an alphabet $\Sigma=\{x_1,\dots,x_d\}$. Since the (standard) growth function $\gamma_{\mathfrak{L}}(n)$ counts words of length $\leq n$, its discrete derivative $\gamma_{\mathfrak{L}}'(n)$ counts words of length $n$. If moreover $\mathfrak{L}$ is prolongable then $\gamma_{\mathfrak{L}}'(n)$ naturally coincides with the complexity function of the underlying subshift.
Since any word of length $n$ can be extended to the right by some letter, and it can be recovered from this extension, we get that the complexity function is non-decreasing. In other words, $\gamma_{\mathcal{L}}''(n)\geq 0$. 
We make the following observation, which we will freely use in the sequel.

\begin{lem} \label{f'' and nf(n)}
Let $f\colon\mathbb{N}\rightarrow\mathbb{N}$. If $f''\geq 0$ then $f(n)\sim nf'(n)$.
\end{lem}

\begin{proof}
Indeed,
$$ f(n) = \sum_{i=1}^{n} f'(i) \leq \sum_{i=1}^{n} f'(n) = nf'(n) $$
and on the other hand:
$$ nf'(n) \leq \sum_{i=n+1}^{2n} f'(i) \leq f(2n), $$
so $f(n)\sim nf'(n)$.
\end{proof}
In particular, this applies when $f(n)$ is the (standard) growth function of a prolongable language.




\subsection{Hierarchy of growth functions} \label{subsec:Hierarchy}

We close this section with an overview of the hierarchy of classes of growth functions, as presented in Figure 1. Though we are interested in growth functions up to asymptotic equivalence, the inclusions and equalities are true even at the level of the functions themselves. Let us prove the non-trivial relations among it. The arrow $(vi)$ is valid only for super-linear growth functions.

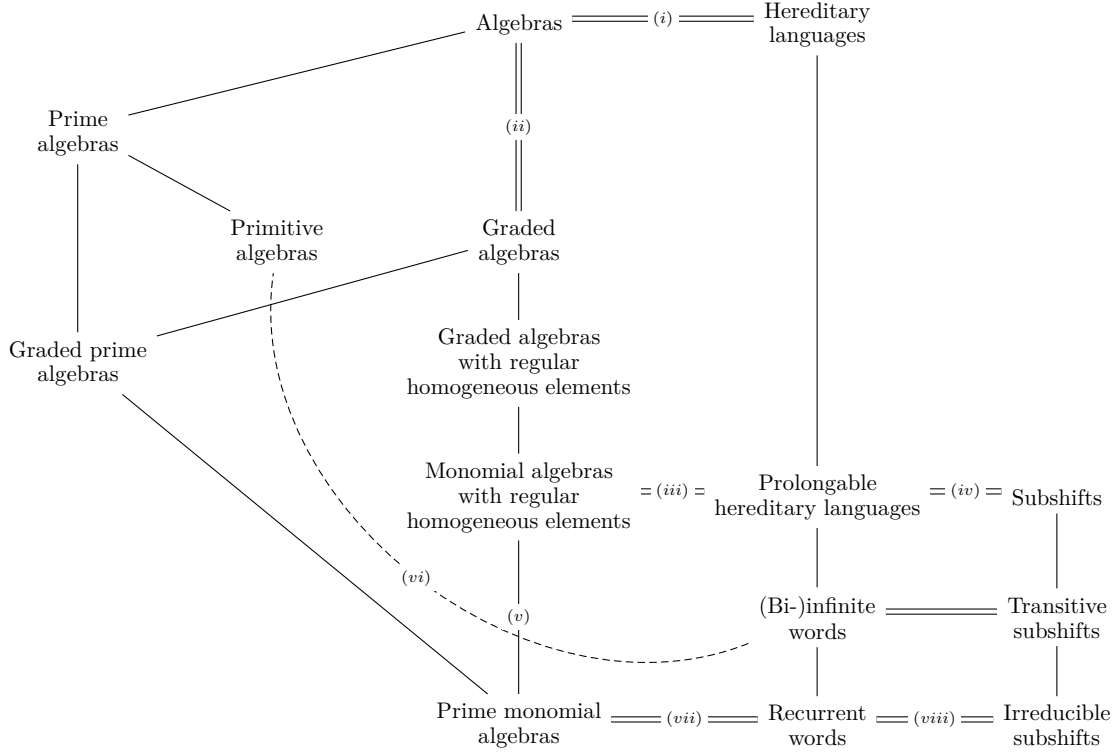
\begin{figure} \label{figure}
\centering
\begin{tikzcd}[scale cd=0.85]
                                                  &                               & \txt{Algebras} \arrow[r, "(i)" description, Rightarrow, no head]                                                          & \txt{Hereditary \\ languages}                                                                                         &                                        \\
\txt{Prime \\ algebras} \arrow[rru, no head]                        &                               &                                                                             &                                                                                                      &                                        \\
                                                  & \txt{Primitive \\ algebras} \arrow[lu, no head] & \txt{Graded \\ algebras} \arrow[uu, "(ii)" description, Rightarrow, no head]                                      &                                                                                                      &                                        \\
\txt{Graded prime \\ algebras} \arrow[rru, no head] \arrow[uu, no head] &                               & \txt{Graded algebras \\ with regular \\ homogeneous elements} \arrow[u, no head]                        &                                                         &                               \\
                                                  &                               &                                     \txt{Monomial algebras \\ with regular \\ homogeneous elements} \arrow[u, no head] \arrow[r, "(iii)" description, Rightarrow, no head]                                        & \txt{Prolongable \\ hereditary languages} \arrow[r, "(iv)" description, Rightarrow, no head] \arrow[uuuu, no head]                                                         & \txt{Subshifts}  \\
                                                  &                               &  & \txt{(Bi-)infinite \\ words} \arrow[r, Rightarrow, no head] \arrow[u, no head] \arrow[lluuu, "(vi)" description, dashed, no head, bend left=60]                                    & \txt{Transitive \\ subshifts} \arrow[u, no head]   \\
                                                  &                               & \txt{Prime monomial \\ algebras} \arrow[uu, "(v)" description, no head] \arrow[r, "(vii)" description, Rightarrow, no head] \arrow[lluuu, no head] & \txt{Recurrent \\ words} \arrow[r, "(viii)" description, Rightarrow, no head] \arrow[u, no head]                                    & \txt{Irreducible \\ subshifts} \arrow[u, no head]             
\end{tikzcd}
\caption{Hierarchy of growth functions: algebra, combinatorics and dynamics. An upward arrow means inclusion of the corresponding classes. The dashed arrow applies to super-linear functions.}
\end{figure}

\begin{proof}
$(i),(ii)$ are shown in Subsections \ref{subsec:algebras} and \ref{subsec:languages and subshifts}.
\\
$(iii)$: Let $\mathcal{L}$ be a prolongable language over a finite alphabet $\Sigma=\{x_1,\dots,x_d\}$. Consider the quotient algebra of the free $F$-algebra generated by $\Sigma$ modulo the ideal generated by all monomials which do not belong to $\mathcal{L}$, say, $A$. Then $A$ is a monomial algebra and thus graded by $\deg(x_1)=\cdots=\deg(x_d)=1$ and $a=x_1+\cdots+x_d$ is a homogneous element of degree $1$. Moreover, it is regular: suppose that $f=\sum_{i=1}^{m}c_iu_i$ is a linear combination of distinct monomials such that $af=0$; we may assume that $u_1,\dots,u_m$ have an equal degree. Thus: $$\sum_{i=1}^{m}\sum_{j=1}^{d} c_ix_ju_i=0.$$ Since $A$ is a monomial algebra and $\{x_ju_i\}_{i=1,j=1}^{m,d}$ are distinct, the coefficients must vanish, so $f=0$. The claim for $fa=0$ is similar.


Conversely, suppose that $A$ is a monomial algebra with a regular homogeneous element $a=\sum_{i=1}^{m}c_iu_i\in A$, where $u_1,\dots,u_m$ are monomials of the same length and their coefficients $c_1,\dots,c_m$ are non-zero.
We claim that the language constituting of all non-zero monomials in $A$ (in the monomial presentation) is a hereditary prolongable language. It is clearly hereditary; to see that it is prolongable, let $u\in A$ be a non-zero monomial. Then, since $a$ is regular, we know that $$\sum_{i=1}^{m} c_iu_iu=au\neq 0,$$ and since $A$ is a monomial algebra, it follows that at least one of the monomials $u_1u,\dots,u_mu$ is non-zero, so $u$ is prolongable to the left. Extension to the right is similar, as $ua\neq 0$.
\\
$(iv)$ is shown in Subsubsection \ref{subsubsec:subshifts correspond to hereditary prolongable languages}.
\\
$(v)$ follows from $(iii)),(vii)$ and the other relations. 
\\
$(vi)$: Let $w\in \Sigma^{\mathbb{Z}}$ and $X=\overline{\{T^iw\}}_{i\in\mathbb{Z}}\subseteq \Sigma^{\mathbb{Z}}$. Let $P(n)$ be the corresponding growth function and $p(n)=P'(n)$ the complexity function. Denote the groupoid of the action $\mathbb{Z}\curvearrowright X$ by $\mathfrak{G}$. If $P(n)$ is linear then it is not equivalent to the growth of any finitely generated primitive algebra, since finitely generated algebras of linear growth are PI \cite{SmallStaffordWarfield}, and finitely generated primitive PI algebras are finite-dimensional \cite{Kaplansky}.
Now, if $P(n)$ is super-linear then $w$ is not eventually periodic. 
Notice that $\mathfrak{G}$ is a Hausdorff groupoid, which is \textit{ample} (namely, the unit space is a locally compact Hausdorff space with a basis consisting of compact open sets) and \textit{effective} (namely, units with trivial isotropy groups are dense in the unit space) since $w$ is a-periodic.
Let $A=F[\mathfrak{G}]$ be the convolution algebra of $\mathfrak{G}$. Therefore, by \cite[Theorem~4.10]{Steinberg}, it follows that the convolution algebra $F[\mathfrak{G}]$ is primitive (clearly $\mathfrak{G}$ has a dense orbit; an explicit faithful simple module can be obtained by the bi-infinite matrix representation presented in \cite{Nekrashevych}). Besides, the growth of the convolution algebra is: $$ \gamma_{F[\mathfrak{G}]}(n)\sim np(n) $$ (see e.g.~\cite[Proposition~4.5]{Nekrashevych}.) But since $P(n)$ is the growth function of a prolongable language, we have:
$$ np(n) \sim P(n). $$
\\
$(vii), (viii)$ are immediate from the definitions together with the observation that a monomial algebra is prime if and only if for any two non-zero monomials $u,v$ there exists $w$ such that $uwv$ is non-zero.
\end{proof}

\section{Realizing and approximating growth functions} \label{sec:realizing}

\subsection{Submultiplicativity suffices up to a linear error term}

In \cite{SmoktunowiczBartholdi}, it is shown that if $f\colon\mathbb{N}\rightarrow\mathbb{N}$ is an increasing, submultiplicative function then there exists a finitely generated graded algebra $A$ such that: $$f(n) \preceq \dim_F A_n \preceq n^2f(n)$$
and consequently:
$$f(n) \preceq \gamma_A(n) \preceq n^3f(n).$$
Building on \cite{BellZelmanov}, we can realize increasing submultiplicative functions up to a \textit{linear} error term, which is the best possible approximation.

\begin{prop} \label{linear error term}
Let $f\colon\mathbb{N}\rightarrow\mathbb{N}$ be an increasing, submultiplicative function. Then there exists a finitely generated algebra $A$ whose growth function satisfies:
$$ f(n)\preceq \gamma_A(n)\preceq nf(n). $$
\end{prop}

\begin{proof}
Let $f\colon\mathbb{N}\rightarrow\mathbb{N}$ be an increasing, submultiplicative function. Let $F(n)=~\sum_{i=1}^{n} f(i)$ be its `discrete integral'. Then, since $f=F'$ is increasing and submultiplicative, in particular $F$ satisfies the conditions of \cite[Theorem~1.1]{BellZelmanov} and therefore there exists a finitely generated algebra $A$ such that $\gamma_A\sim F$. But $f(n)\leq F(n)\leq nf(n)$, and the claim follows.
\end{proof}

\begin{rem}
The above approximation is best possible: for any $\varepsilon > 0$ one can find a submultiplicative and increasing function $f\colon \mathbb{N}\rightarrow \mathbb{N}$ such that $f(n)/n^{1+\frac{\varepsilon}{2}}\xrightarrow{n\rightarrow \infty}~1$. Thus any function squeezed between $f(n)$ and $n^{1-\varepsilon} f(n)$ is super-linear and sub-quadratic, so cannot be equivalent to the growth of any algebra by Bergman's gap theorem.
\end{rem}


\subsection{Prime monomial algebras}

We start with the following construction, which originates from \cite{SmoktunowiczBartholdi} and modified in \cite{BGIsr}, brought here in an improved version.

Let $f\colon \mathbb{N}\rightarrow \mathbb{N}$ be a non-decreasing function satisfying $f(2^{n+1})\leq f(2^n)^2$. Let $X$ be a set of cardinality $f(1)$.
Define a sequence of integers $c_{2^n}$ for $n\geq 0$ as follows. First, $c_1=\lceil \frac{f(2)}{f(1)} \rceil$. For $n\geq 1$:
\[ c_{2^n} = \left\{
\begin{array}{ll}
      \lceil \frac{f(2^{n+1})}{f(2^n)} \rceil & \text{if}\  f(1)c_1c_2\cdots c_{2^{n-1}}<2f(2^n) \\
      \lfloor \frac{f(2^{n+1})}{f(2^n)} \rfloor & \text{if}\  f(1)c_1c_2\cdots c_{2^{n-1}}\geq 2f(2^n)
\end{array} \right. \]

\begin{lem} \label{lemma_c}
We have: $$f(2^{n+1})\leq f(1)c_1c_2\cdots c_{2^n}\leq 4f(2^{n+1})$$
\end{lem}

\begin{proof}
We prove this by induction.  For $n=0$, notice that $f(2)\leq f(1)c_1$ and $f(1)c_1\leq f(2)+f(1)\leq 2f(2)\leq 4f(2)$.

Now for the induction step, assume that $f(2^n)\leq f(1)c_1c_2\cdots c_{2^{n-1}}\leq 4f(2^n)$. If $f(1)c_1c_2\cdots c_{2^{n-1}}<2f(2^n)$ then:
\begin{eqnarray*}
f(1)c_1c_2\cdots c_{2^{n-1}}c_{2^n} & < & 2f(2^n)\left(\frac{f(2^{n+1})}{f(2^n)}+1\right) \\ & = & 2f(2^{n+1})+2f(2^n)\leq 4f(2^{n+1})
\end{eqnarray*}
and:
$$ f(2^{n+1})=f(2^n)\cdot \frac{f(2^{n+1})}{f(2^n)}\leq  f(1)c_1c_2\cdots c_{2^{n-1}}c_{2^n} $$
whereas if $f(1)c_1c_2\cdots c_{2^{n-1}}\geq 2f(2^n)$ then: $$ f(1)c_1c_2\cdots c_{2^{n-1}}c_{2^n} \leq 4f(2^n)\cdot \frac{f(2^{n+1})}{f(2^n)}=4f(2^{n+1}) $$
and, if furthermore $f(2^{n+1})\geq 2f(2^n)$:
\begin{eqnarray*}
f(1)c_1c_2\cdots c_{2^{n-1}}c_{2^n} & \geq & f(1)c_1c_2\cdots c_{2^{n-1}}\left(\frac{f(2^{n+1})}{f(2^n)}-1\right) \\ & \geq & 2f(2^{n+1})-2f(2^n)\geq f(2^{n+1})
\end{eqnarray*}
otherwise if $f(2^{n+1})<2f(2^n)$ then:
$$ f(1)c_1c_2\cdots c_{2^{n-1}}c_{2^n} \geq f(1)c_1c_2\cdots c_{2^{n-1}} \geq 2f(2^n) \geq f(2^{n+1}) $$
as claimed.
\end{proof}

Define a sequence of sets $W(2^n),C(2^n)$ of finite words over $X$, along with a sequence of finite sequences of words $U(2^n)\subseteq X^*$ over $X$ as follows. Let $W(1)=X$ (recall that $|X|=f(1)$).
Suppose that $W(2^n)\subseteq~X^{2^n}$ and $U(2^n)=\{u_1,\dots,u_l\}$ is given.

We pick a subset $C(2^n)\subseteq W(2^n)$ of cardinality $c_{2^n}$, such that $C(2^n)$ contains at least one word with $u_1$ as a prefix. Let $W(2^{n+1})=W(2^n)C(2^n)$. We define $U(2^{n+1})$ as follows: append to (the end of) $U(2^n)$ the words of $W(2^{n+1})$, and erase $u_1$; rename the words of $U(2^{n+1})=\{u_1,\dots,u_{l'}\}$.
This process is indeed possible by Lemma \ref{lemma_c}, since: 
\begin{eqnarray*}
c_{2^n}\leq \lceil \frac{f(2^{n+1})}{f(2^n)} \rceil \leq \lceil f(2^n) \rceil & \leq &  f(1)c_1c_2\cdots c_{2^{n-1}} \\ & = & |W(1)C(2)C(2^2)\cdots C(2^{n-1})|=|W(2^n)|
\end{eqnarray*}

Consider the following set of right-infinite words:
$$ \mathcal{S} = W(1)C(1)C(2)C(2^2)C(2^3)C(2^4)\cdots \subseteq X^{\infty} $$
and let $A_{\mathcal{S}}$ be the quotient of the free algebra $F\left<X\right>$ modulo the ideal generated by all monomials which do not occur as factors of any word from $\mathcal{S}$.

\begin{lem} \label{primeness of S}
The algebra $A_{\mathcal{S}}$ is prime.
\end{lem}
\begin{proof} 
By construction, every word $u\in W(2^n)$ appears in $U(2^n)$, and therefore for some $N>n$ we have a word in $C(2^N)$ with $u$ as a prefix. It follows that $w$ factors words from infinitely many $C(2^{n_i})$'s. By construction of $\mathcal{S}$, it follows that for every pair of non-zero monomials $w_1,w_2\in A_{\mathcal{S}}$ there exists a monomial $v$ such that $w_vw_2$ factors an infinite word from $\mathcal{S}$.
\end{proof}

Let $\gamma_{\mathcal{S}}(n)$ be he number of words of length at most $n$ which occur as factors of some infinite word in $\mathcal{S}$. This function coincides with the growth function of $A_{\mathcal{S}}$ with respect to the standard set of generators $X\cup \{1\}$.

\begin{lem} \label{SBmon}
Under the above notations, there exists $C,D>0$ such that: $$\gamma_{\mathcal{S}}'(n)\leq Cnf(Dn),\ \ \gamma_{\mathcal{S}}(n) \leq Cn^2f(Dn)$$
and $f(n)\leq \gamma_{\mathcal{S}}'(2n)$.
\end{lem}

\begin{proof}
By \cite[Lemma~6.3]{SmoktunowiczBartholdi}, every length-$2^m$ word which factors $\mathcal{S}$ is a subword of $W(2^m)C(2^m)\cup C(2^m)W(2^m)$. Each word in $W(2^m)C(2^m)\cup C(2^m)W(2^m)$ has length $2^{m+1}$, so the number of length-$2^m$ words, which is equal to $\gamma_{\mathcal{S}}'(2^m)=\gamma_{\mathcal{S}}(2^m)-\gamma_{\mathcal{S}}(2^m-1)$, is at least $|W(2^m)|$ and at most $$2^m \cdot |\left(W(2^m)C(2^m)\cup C(2^m)W(2^m)\right)|$$
Now since $f(2^m)\leq |W(2^m)|\leq 4f(2^m)$, we have:
$$ f(2^m) \leq \gamma_{\mathcal{S}}'(2^m) \leq 2^m \cdot \left(2\cdot 4f(2^m) \cdot \lceil \frac{f(2^{m+1})}{f(2^m)} \rceil \right) \leq 2^{m+4}f(2^{m+1}) $$
Since every monomial in $\mathcal{S}$ is extendable, $\gamma_{\mathcal{S}}'$ is non-decreasing, so $f(2^m)\leq \gamma_{\mathcal{S}}(2^m)\leq 2^{2m+4}f(2^{m+1})$. Since $f$ is non-decreasing as well, for every $n$, if we take $2^m\leq n\leq 2^{m+1}$ then we obtain that $f(n)\leq \gamma_{\mathcal{S}}'(2n)$ and $\gamma_{\mathcal{S}}'(n)\leq 64nf(4n)$ and $\gamma_{\mathcal{S}}(n)\leq 64n^2f(4n)$.
\end{proof}

\begin{prop} \label{prime quadratic}
Let $f\colon \mathbb{N}\rightarrow \mathbb{N}$ be an increasing submultiplicative function. Then there exists a finitely generated prime monomial algebra $A$ whose growth satisfies:
$$ f(n)\preceq \gamma_A(n) \preceq n^2 f(n) $$
\end{prop}
\begin{proof} 
This immediately follows from Lemma \ref{primeness of S} and Lemma \ref{SBmon}.
\end{proof}

Prime monomial algebras give rise to hereditary prolongable languages and correspond to recurrent bi-infinite words or equivalently irreducible subshifts (see Section \ref{sec:roadmap}).

\begin{cor} \label{cor}
Let $f:\mathbb{N}\rightarrow \mathbb{N}$ be an increasing submultiplicative function. Then there exists an irreducible subshift $X\subseteq \Sigma^{\mathbb{Z}}$ whose complexity function satisfies:
$$f(n) \preceq p_X(n) \preceq n f(n)$$
Consequently, for the growth function $\gamma_X(n)$ of the subshift $X$, we have:
$$ f(n) \preceq \gamma_X(n) \preceq n^2 f(n). $$
\end{cor}

\begin{proof}
This follows from Proposition \ref{prime quadratic} and Lemma \ref{primeness of S}, since prime monomial algebras give rise to recurrent bi-infinite words (see Subsection \ref{subsec:Hierarchy}), and hence to irreducible subshifts with the same growth functions; recall that $\gamma_{\mathcal{S}}'=p_X$ for $X$ being the subshift corresponding to the prime monomial algebra $A_{\mathcal{S}}$.
\end{proof}

\subsection{Primitive algebras}

\begin{prop}[{Proposition \ref{primitive_proposition}}] \label{primitive realization pre} 
Let $f\colon \mathbb{N}\rightarrow \mathbb{N}$  be an increasing submultiplicative function. Then there exists a finitely generated primitive algebra $B$ such that:
We have: $$f(n)\preceq \gamma_B(n)\preceq n^2f(n)$$
In particular, this applies to $f=\gamma_A$ the growth function of an arbitrary algebra.
\end{prop}

\begin{proof} 

By Proposition \ref{prime quadratic} there exists a finitely generated, prime monomial algebra $A$ such that $f(n)\preceq \gamma_A(n)\preceq n^2f(n)$. As proven in Subsection \ref{subsec:Hierarchy}, there exists a finitely generated primitive algebra $B$ such that $\gamma_A\sim \gamma_B$ (we may obviously assume that $A$ does not have linear growth). The claim follows.
\end{proof}

\section{Arbitrarily rapid holes in the space of growth functions} \label{sec:arbitrarily}

\subsection{Asymptotic properties of growth functions}

By a result of Bell and Zelmanov \cite[Proposition~2.1]{BellZelmanov}, if $\gamma$ is a growth function of a finitely generated algebra then $\gamma'(m)\leq \gamma'(n)^2$ for every $m\in\{n,\dots,2n\}$. 
Their proof technique can be modified to yield:
\begin{rem}
Let $\gamma$ be the growth function of a finitely generated algebra with respect to some generating subspace. Let $d \in \mathbb{N}$. Then 
$\gamma'(m)\leq \gamma'(n)^d$ for every $m\in\{n,\dots,dn\}$.
\end{rem}
\begin{proof}
We may assume the algebra is monomial, so $\gamma'(n)$ is the number of (nonzero) words of length $n$ in the generators.
But if $n \leq m \leq dn$ then every word of length $m$ is a prefix of a product of $d$ words of length $n$, from which it can be recovered.
\end{proof}

This is utilized to construct an obstruction for a function to be realizable, \textit{up to equivalence}, as a growth function of an algebra.

\begin{prop} \label{property_growth}
Suppose that $f\colon\mathbb{N}\rightarrow \mathbb{N}$ is equivalent to the growth function $\gamma\colon\mathbb{N}\rightarrow \mathbb{N}$ of a finitely generated algebra with respect to a generating subspace. Then there exists $C\in \mathbb{N}$ such that for all $D\geq C^2$, for all $n$ we have: $$f(2CDn)-f(2CDn-C)\leq 2D^2n (f(CDn)-f(Cn-C))^{2D}.$$
\end{prop}

\begin{proof}
Since $f\sim \gamma$, there exists $C>0$ such that $f(n)\leq \gamma(Cn)$ and $\gamma(n)\leq f(Cn)$. In particular, for every $D\geq C^2$, we have $\gamma(n)\leq f(Cn)\leq \gamma(Dn)$.
Set $h(n)=f(Cn)$ and $\varphi(n)=\gamma(Dn)$. Then $h(n)\leq \varphi(n)\leq h(Dn)$.
Observe that: $$h'(n)=f(Cn)-f(Cn-C)\leq \gamma(Dn)-\gamma(n-1)=\sum_{k=n}^{Dn} \gamma'(k)\leq Dn\gamma'(n)^D.$$
Note also that: 
\begin{eqnarray*}
\gamma'(Dn) & =& \gamma(Dn)-\gamma(Dn-1) \\ 
& \leq & \gamma(Dn)-\gamma(Dn-D) \\
& = & \varphi(n)-\varphi(n-1) \\
& \leq & h(Dn)-h(n-1).
\end{eqnarray*}
Putting these together, we get that:
\begin{eqnarray*}
f(2CDn)-f(2CDn-C) & = & h'(2Dn) \\
& \leq & 2D^2n\gamma'(2Dn)^D 
\\
& \leq & 2D^2n\gamma'(Dn)^{2D} \\ 
& \leq & 2D^2n (h(Dn)-h(n-1))^{2D} \\ & = & 2D^2n (f(CDn)-f(Cn-C))^{2D},
\end{eqnarray*}
as desired.
\end{proof}

\subsection{Constructing a submultiplicative function}\label{construction of f}

Let $1 < d_1<d_2<\cdots$ be an increasing sequence of integers, and $n_1,n_2,\dots$ a sequence of positive integers such that: 
$$n_1 < d_1 n_1 < n_2 < d_2 n_2 < n_3 < \cdots.$$ 
Both sequences are to be restricted in the sequel
by conditions of the form ``$d_k$ is greater than a function of $\{d_i,n_i\}_{i=1}^{k-1}$'' and ``$n_k$ is greater than a function of $\{d_i,n_i\}_{i=1}^{k-1}$ and $d_k$''.

\subsubsection{The interval $[1,n_2]$}
We will define a function $f\colon\mathbb{N}\rightarrow \mathbb{N}$, first by defining it on the domain $[1,n_2]$:
\begin{itemize}
\item For $x\leq n_1$, take $f(x)=2^x$;
\item For $n_1<x\leq d_1n_1$, take $f(x)=f(x-1)+x+1$;
\item For $d_1n_1<x\leq n_2$, take $f(x)=\lfloor 2^{1/{2d_1}}f(x-1)\rfloor$.
\end{itemize}

Denote $\alpha_1=f(d_1n_1)-f(n_1)
< d_1^2n_1^2$. Since $\alpha_1$ is polynomial with respect to $n_1$ (assuming $d_1$ was fixed), if we take $n_1\gg 1$ then we may assume that $f(d_1n_1) =~2^{n_1}+\alpha_1\leq 2^{n_1+\frac{1}{3}}$ and $f'(x)\geq x+1$ under the above definition.
We also need the following fact:

\begin{lem}\label{floor}
Given $c>1$ and $\varepsilon>0$, for all $a_0\gg 1$ the sequence $a_{k+1}=\lfloor ca_k \rfloor$ satisfies $c^{k-\varepsilon}a_0\leq a_k\leq c^ka_0$.
\end{lem}

\begin{proof}
Obviously $a_{k+1} \leq c a_k$ so $a_k \leq c^k a_0$. By induction $a_k\geq c^ka_0-\frac{c^k-1}{c-1}$, so: 
$$a_k-c^{k-\varepsilon}a_0\geq (c^k-c^{k-\varepsilon})a_0-\frac{c^k-1}{c-1}\xrightarrow{a_0\rightarrow \infty}\infty.$$
\end{proof}

Using Lemma \ref{floor} (taking $c = 2^{\frac{1}{2d_1}}$, $\varepsilon = 2^{-3}$), we can take $n_1\gg 1$ so that if $x\geq d_1n_1$ then $f(x)\geq f(d_1n_1)\cdot 2^{\frac{x-d_1n_1}{2d_1}-2^{-3}}$.
It is evident that $f$ is increasing in $[1,d_1n_1]$. 
We now turn to prove that $f$ is submultiplicative.

\begin{prop}\label{f submul 0}
If $p,q\in \mathbb{N}$ satisfy $p+q\leq n_2$ then $f(p+q)\leq f(p)f(q)$.
\end{prop}
\begin{proof}
We first take care of the interval $[d_1n_1+1,n_2]$. Pick $p+q>d_1n_1$ with $p\leq q$ and we must show that $f(p+q)\leq f(p)f(q)$. We assume that $d_1>2$ and $n_1\gg 1$ (by means to be specified along the proof) and compute that:
\begin{eqnarray*}
f(p+q)&\leq &f(d_1n_1)\cdot 2^{\frac{p+q-d_1n_1}{2d_1}}\\
&\leq & 2^{n_1+\frac{1}{3}+\frac{p+q-d_1n_1}{2d_1}}\\
&=& 2^{\frac{1}{2}n_1+\frac{p+q}{2d_1}+\frac{1}{3}}.
\end{eqnarray*}
Assume that $p\leq n_1$, then $q>n_1$. Whether or not $q \leq d_1n_1$, we have that: $$\frac{f(p+q)}{f(q)}\leq 2^{\frac{p}{2d_1}}\leq 2^p=f(p);$$ the first inequality follows since the ratio between two successive numbers in $[n_1,n_2]$ is $\leq 2^{\frac{1}{2d_1}}$ if we only take $n_1\gg 1$. Thus we suppose $n_1<p$ (so in particular $f(q)\geq f(p)\geq 2^{n_1}$).
\begin{itemize}
    \item If $d_1n_1\leq p$ then (assuming $n_1>1$): 
    \begin{eqnarray*}
    f(p)f(q)&\geq & 2^{n_1+\frac{p-d_1n_1}{2d_1}-2^{-3}}2^{n_1+\frac{q-d_1n_1}{2d_1}-2^{-3}}\\
    &=& 2^{n_1+\frac{p+q}{2d_1}-2^{-2}}\\ 
    &\geq & 2^{\frac{1}{2}n_1+\frac{p+q}{2d_1}+\frac{1}{3}}\geq f(p+q).
    \end{eqnarray*}
    \item If $q\leq d_1n_1$ then: $$f(p)f(q)\geq 2^{2n_1}\geq 2^{\frac{1}{2}n_1+n_1+\frac{1}{3}}\geq f(p+q),$$ the latter inequality follows since $p+q\leq 2d_1n_1$.
    \item If $p<d_1n_1<q$ then: \begin{eqnarray*}
    f(p)f(q)& \geq &  2^{n_1}2^{n_1+\frac{q-d_1n_1}{2d_1}-2^{-3}} \\ & =& 2^{\frac{3}{2}n_1+\frac{q}{2d_1}-2^{-3}} \\ & = &  2^{\frac{3}{2}n_1+\frac{p+q}{2d_1}-\frac{p}{2d_1}-2^{-3}}\\
    & \geq & 2^{\frac{3}{2}n_1+\frac{p+q}{2d_1}-\frac{1}{2}n_1-2^{-3}}>f(p+q).
    \end{eqnarray*}
\end{itemize}
(The penultimate inequality follows since $p<d_1n_1$.) As for submultiplicativity in the interval $[n_1,d_1n_1]$ (note that the case $p+q\in [1,n_1]$ is trivial as the function $x\mapsto 2^x$ is submultiplicative), assume that $n_1\leq p+q\leq d_1n_1$.
\begin{itemize}
    \item If $p>n_1$ then: $$f(p)f(q)\geq f(n_1)^2=2^{2n_1}\geq 2^{n_1+\frac{1}{3}}=f(d_1n_1)\geq f(p+q).$$
    \item If $q<n_1$ then: $$f(p)f(q)=2^{p+q}\geq f(p+q).$$
    \item If $p\leq n_1\leq q$ then:
    \begin{eqnarray*}
    f(p+q) & \leq & f(d_1n_1) \\
    &\leq & 2^{n_1+\frac{1}{3}}\leq 2^p2^{n_1}=f(p)f(n_1)\leq f(p)f(q).
    \end{eqnarray*}
\end{itemize}
And the claim follows. 
\end{proof}

\subsubsection{Extending $f$ to $\mathbb{N}$}
We now extend $f$ to $\mathbb{N}$ as follows. Suppose that $d_1,\dots,d_{k-1},\\ n_1,\dots,n_{k-1}$ are fixed and $f$ is defined on the domain $[1,n_k]$ (we choose $n_k$ only after $\{d_i,n_i\}_{i=1}^{k-1},d_k$ were fixed).
Assumptions on $d_k,n_k$ will be made explicitly during the proof of submultiplicativity, in order to clarify where these assumptions originate from.
We assume $d_k\geq \frac{n_{k-1}}{2d_1\cdots d_{k-2}}$.
Define:
\begin{itemize}
    \item For $n_k<x\leq d_kn_k$, take $f(x)=f(x-1)+x+1$;
    \item For $d_kn_k<x\leq n_{k+1}$, take $f(x)=\lfloor 2^{\frac{1}{2d_1\cdots d_k}}f(x-1) \rfloor$.
\end{itemize}
Note that by taking $n_k$ to be large enough we can make sure that $f'(x)\geq x+1$.

\textbf{Condition (I).} We pick $n_k$ large enough such that for all $d_kn_k\leq x\leq y \leq n_{k+1}$ we have that: $$f(y)\geq f(x)\cdot 2^{\frac{y-x}{2d_1\cdots d_k}-2^{-k-2}}$$ (This is possible by Lemma \ref{floor} applied with $c=2^{\frac{1}{2d_1\cdots d_k}}$ and $\varepsilon=2^{-k-2}$.)
In particular,
$$f(x)\geq f(d_kn_k)\cdot 2^{\frac{x-d_kn_k}{2d_1\cdots d_k}-2^{-k-2}}$$

\begin{lem} \label{lower bound f}
We can choose $\{n_k\}_{k=1}^{\infty}$ to be sufficiently sparse such that for every $2\leq x\leq~d_kn_k$ we have that: $$f(x)\geq 2^{\frac{x}{2d_1\cdots d_k}+1+2^{-k-1}}$$
\end{lem}
\begin{proof}
We prove the assertion by induction on $k$. For $k=1$, if $2\leq x\leq n_1$ then: $$f(x)=2^x\geq 2^{\frac{x}{2d_1}+1+2^{-2}}$$
and if $n_1<x\leq d_1n_1$ then: $$f(x)\geq 2^{n_1}\geq 2^{\frac{d_1n_1}{2d_1}+1+2^{-2}}\geq 2^{\frac{x}{2d_1}+1+2^{-2}}$$ so the assertion holds (indeed, we take $n_1\geq 3$). Suppose the claim holds for $2\leq x\leq d_kn_k$ and let us prove it for $2\leq x\leq d_{k+1}n_{k+1}$; if $x\leq d_kn_k$ this is immediate from the induction hypothesis. If $d_kn_k<x\leq n_{k+1}$ then by Condition (I): $$f(x)\geq f(d_kn_k)\cdot 2^{\frac{x-d_kn_k}{2d_1\cdots d_{k}}-2^{-k-2}}$$ We can bound the latter term from below (using what we already know about $d_kn_k$):
\begin{eqnarray*}
f(d_kn_k)\cdot 2^{\frac{x-d_kn_k}{2d_1\cdots d_{k}}-2^{-k-2}}& \geq & 2^{\frac{d_kn_k}{2d_1\cdots d_{k}}+1+2^{-k-1}}\cdot 2^{\frac{x-d_kn_k}{2d_1\cdots d_{k}}-2^{-k-2}}\\ 
&= & 2^{\frac{x}{2d_1\cdots d_{k}}+1+2^{-k-2}}
\end{eqnarray*}
If $n_{k+1}<x\leq d_{k+1}n_{k+1}$ then (using what we already know about $n_{k+1}$): 
\begin{eqnarray*}
f(x)\geq f(n_{k+1})& \geq & 2^{\frac{n_{k+1}}{2d_1\cdots d_{k}}+1+2^{-k-2}}\\
&=& 2^{\frac{d_{k+1}n_{k+1}}{2d_1\cdots d_{k+1}}+1+2^{-(k+1)-1}}\\
&\geq & 2^{\frac{x}{2d_1\cdots d_{k+1}}+1+2^{-(k+1)-1}},
\end{eqnarray*}
as desired.
\end{proof}
We will use the following freely:
\begin{lem}
Given $\varepsilon>0$ we can choose $\{d_k,n_k\}_{k=1}^{\infty}$ such that $f(d_kn_k)\leq f(n_k)\cdot{2^\varepsilon}$. 
\end{lem}
\begin{proof}
Using Lemma \ref{lower bound f} we can make sure that: $$(2^\varepsilon - 1)f(n_k)\geq (2^\varepsilon - 1)2^{\frac{n_k}{2d_1\cdots d_{k}}}\geq d_k^2n_k^2\geq f(d_kn_k)-f(n_k).$$
where the middle inequality follows since as $d_1,\dots,d_k$ are fixed, the left hand side grows more rapidly than the right hand side. The last inequality follows from the definition of $f$ on $[n_k,d_kn_k]$,
and the claim follows.
\end{proof}
We take $n_k$ large enough such that $f(d_kn_k)\leq f(n_k)\cdot 2^{\frac{1}{3}}$.
\begin{prop}\label{f submul}
The function $f\colon\mathbb{N}\rightarrow \mathbb{N}$ constructed above is submultiplicative.
\end{prop}
\begin{proof}
Let $p,q\in \mathbb{N}$ be such that $n_k+1\leq p+q\leq n_{k+1}$ and let us prove that $f(p+q)\leq f(p)f(q)$. We work by induction, where the basis $k\leq 2$ follows from Proposition \ref{f submul 0}.

We begin with $p+q\in [d_kn_k+1,n_{k+1}]$ (without limiting $n_{k+1}$, which can be thought of as infinity). Without loss of generality, $p\leq q$. Denote $\beta=f(d_{k-1}n_{k-1})$. Then: 

\begin{eqnarray*}
f(p+q) & \leq & 2^{\frac{p+q-d_kn_k}{2d_1\cdots d_k}}f(d_kn_k) \\ & \leq & 2^{\frac{p+q-d_kn_k}{2d_1\cdots d_k}+\frac{1}{3}}f(n_k) \\ & \leq & 2^{\frac{p+q-d_kn_k}{2d_1\cdots d_k}+\frac{1}{3}}\cdot \beta\cdot 2^{\frac{n_k-d_{k-1}n_{k-1}}{2d_1\cdots d_{k-1}}} \\ & = & \beta\cdot 2^{\frac{p+q-d_{k-1}d_{k}n_{k-1}}{2d_1\cdots d_k}+\frac{1}{3}}
\end{eqnarray*}
We divide into cases:
\begin{itemize}
    \item \textbf{Suppose that $p\leq d_{k-1}n_{k-1}$.} Notice that for any $t\in [n_k,d_kn_k-1]$ we have that: $$(\#)\ \ \ \ \ \frac{f(t+1)}{f(t)}\leq 1+\frac{t+2}{f(t)}\leq 1+\frac{d_kn_k+1}{2^{\frac{n_k}{2d_1\cdots d_k}}}$$ which we can take to be smaller than $2^{\frac{1}{2d_1\cdots d_k}}$ by taking $n_k$ to be large enough. Thus, if in addition we take $d_k\geq \frac{n_{k-1}}{2d_1\cdots d_{k-2}}$ then: $$\frac{f(p+q)}{f(q)}\leq 2^{\frac{p}{2d_1\cdots d_k}}\leq 2^{\frac{d_{k-1}n_{k-1}}{2d_1\cdots d_k}}\leq 2\leq f(p).$$
    (Note that the first inequality is evident if $q\geq d_kn_k$, and otherwise follows from the argument in the beginning of this case.)
    \item \textbf{If $d_{k-1}n_{k-1}<p\leq n_k$ then $q\geq n_k$ (as $d_k>2$), and assume in addition that $q\leq d_kn_k$.} 
    Note also that we can choose $d_k>d_{k-1}n_{k-1}+1$ so: $$d_{k-1}d_kn_{k-1}\leq (d_{k-1}n_{k-1}+1)(d_k-1)\leq p(d_k-1)$$ and thus (recalling that $q\leq d_kn_k$):
    \begin{eqnarray*}
    (\star)\ \ \ pd_k+d_kn_k-2d_{k-1}d_kn_{k-1}&\geq &  pd_k-p(d_k-1)-d_kd_{k-1}n_{k-1}+d_kn_k \\
    &\geq & p+q-d_{k-1}d_kn_{k-1}.\ \ \ \ \ 
    \end{eqnarray*}
    Then, using Condition (I): 
    \begin{eqnarray*}
    f(p)f(q)&\geq &f(p)f(n_k) \\
    &\geq &\beta^2\cdot 2^{\frac{p-d_{k-1}n_{k-1}}{2d_1\cdots d_{k-1}}-2^{-(k-1)-2}+\frac{n_k-d_{k-1}n_{k-1}}{2d_1\cdots d_{k-1}}-2^{-(k-1)-2}} \\
    &=&\beta^2\cdot 2^{\frac{pd_k+n_kd_k-2d_kd_{k-1}n_{k-1}}{2d_1\cdots d_{k}}-2^{-k}} \\
    &\geq & \beta\cdot 2^{\frac{p+q-d_kd_{k-1}n_{k-1}}{2d_1\cdots d_{k}}+\frac{1}{3}} \\
    &\geq &f(p+q).
    \end{eqnarray*}
    (The one before last inequality follows from $(\star)$ combined with the fact that $\beta \geq 2$.)
    \item \textbf{If $d_{k-1}n_{k-1}<p\leq n_k$ then $q\geq n_k$, and now assume that moreover $q>d_kn_k$.}
    By Lemma \ref{lower bound f} we have: $$\frac{f(p+q)}{f(q)}\leq 2^{\frac{p}{2d_1\cdots d_k}}\leq f(p).$$
\end{itemize}
    For the remaining cases, let $p>n_k$. 
\begin{itemize}
    \item \textbf{If $p\geq d_kn_k$} then by Condition (I):
    \begin{eqnarray*}
    f(p)f(q) & \geq & f(d_kn_k)^2\cdot 2^{\frac{p+q-2d_kn_k}{2d_1\cdots d_k}-2\cdot 2^{-k-2}},\\
    f(p+q) & \leq & f(d_kn_k)\cdot 2^{\frac{p+q-d_kn_k}{2d_1\cdots d_k}}.
    \end{eqnarray*}
    Therefore:
    \begin{eqnarray*}
    \frac{f(p)f(q)}{f(p+q)}& \geq & f(d_kn_k)\cdot 2^{-\frac{d_kn_k}{2d_1\cdots d_k}-2^{k-1}}\\
    &\geq & 2^{\frac{d_kn_k}{2d_1\cdots d_k}+1+2^{-k-1}}\cdot 2^{-\frac{d_kn_k}{2d_1\cdots d_k}-2^{k-1}}>1.
    \end{eqnarray*}
    (The middle inequality follows by Lemma \ref{lower bound f}.)
    
    
    
    \item \textbf{If $n_k\leq p< d_kn_k$ and $q\leq d_kn_k$} then: $$f(p+q)\leq \beta \cdot 2^{\frac{p+q-d_kd_{k-1}n_{k-1}}{2d_1\cdots d_k}+\frac{1}{3}} \leq \beta \cdot 2^{\frac{2d_kn_k-d_kd_{k-1}n_{k-1}}{2d_1\cdots d_k}+\frac{1}{3}}$$ and by Condition (I) specified for $x=n_k$:
    \begin{eqnarray*}
    f(p)f(q)\geq f(n_k)^2 & \geq & \beta^2\cdot 2^{\frac{2n_k-2d_{k-1}n_{k-1}}{2d_1\cdots d_{k-1}}-2\cdot 2^{-(k-1)-2}}\\
    &=& \beta^2\cdot 2^{\frac{2d_kn_k-2d_kd_{k-1}n_{k-1}}{2d_1\cdots d_k}-2^{-k}},
    \end{eqnarray*}
    and by Lemma \ref{lower bound f} applied for $x=d_{k-1}n_{k-1}$: $$\beta\geq 2^{\frac{d_{k-1}n_{k-1}}{2d_1\cdots d_{k-1}}+1}=2^{\frac{d_kd_{k-1}n_{k-1}}{2d_1\cdots d_{k}}+1}$$ so:
    $$f(p)f(q)\geq \beta\cdot 2^{\frac{2d_kn_k-d_kd_{k-1}n_{k-1}}{2d_1\cdots d_k}+1-2^{-k}}\geq f(p+q).$$
    \item \textbf{If $n_k<p<d_kn_k<q$} then, applying Lemma \ref{lower bound f} on $x=p$, and Condition (I) for $x=q$: \begin{eqnarray*}
    f(p)f(q)&\geq & 2^{\frac{p}{2d_1\cdots d_k}+1}f(q) \\
    &\geq & 2^{\frac{p}{2d_1\cdots d_k}+1}\cdot f(d_kn_k)\cdot 2^{\frac{q-d_kn_k}{2d_1\cdots d_k}-2^{-k-2}}\\
    & \geq & f(d_kn_k)\cdot 2^{\frac{p+q-d_kn_k}{2d_1\cdots d_k}} \\
    &\geq & f(p+q).
    \end{eqnarray*}
\end{itemize}
We thus proved that $f(p+q)\leq f(p)f(q)$ whenever $p+q\in [d_kn_k+1,n_{k+1}]$. It remains to show that $f(p+q)\leq f(p)f(q)$ for $p+q\in [n_k+1,d_kn_k]$.
\begin{itemize}
    \item \textbf{If $p>n_k$} then, applying Lemma \ref{lower bound f} for $x=n_k$:
    \begin{eqnarray*}
    f(p)f(q)\geq f(n_k)^2 & \geq & f(n_k)\cdot 2^{\frac{n_k}{2d_1\cdots d_k}+1}\\ & \geq & f(n_k)+d_k^2n_k^2\geq f(d_kn_k)\geq f(p+q),
    \end{eqnarray*}
    where the inequality $$f(n_k)\cdot 2^{\frac{n_k}{2d_1\cdots d_k}+1}\geq f(n_k)+d_k^2n_k^2$$ follows since if $d_k$ is fixed then the left hand side grows, as a function of $n_k$, more rapidly than the right hand side, so in particular we can take $n_k$ to be large enough such that this inequality holds.
    \item \textbf{If $p\leq q<n_k$} then: 
    \begin{eqnarray*}
    f(p+q)& \leq & f(n_k)+(p+q-n_k)^2\\
    & \leq & f(n_k)+n_k^2\\
    & \leq & f(n_k)\cdot 2^{\frac{1}{2}}\\
    & \leq & \beta\cdot 2^{\frac{n_k-d_{k-1}n_{k-1}}{2d_1\cdots d_{k-1}}+\frac{1}{2}}.
    \end{eqnarray*}
    (The penultimate inequality follows since $f$ grows exponentially in the interval $[d_{k-1}n_{k-1}+1,n_k]$, so in particular we can take $n_k$ to be large enough such that $f(n_k)\gg n_k^2$).
    \begin{itemize}
\item  If in addition $d_{k-1}n_{k-1}<p$ then, using Condition (I): 
    \begin{eqnarray*}
    f(p)f(q)&\geq & \beta ^2\cdot 2^{\frac{p+q-2d_{k-1}n_{k-1}}{2d_1\cdots d_{k-1}}-2\cdot 2^{-(k-1)-2}}\\
    &=& \beta^2\cdot 2^{\frac{p+q-d_{k-1}n_{k-1}}{2d_1\cdots d_{k-1}}}\cdot 2^{\frac{-d_{k-1}n_{k-1}}{2d_1\cdots d_{k-1}}-2^{-k}}.
    \end{eqnarray*}
    But applying Lemma \ref{lower bound f} for $x=d_{k-1}n_{k-1}$ we get $\beta\geq 2^{\frac{d_{k-1}n_{k-1}}{2d_1\cdots d_{k-1}}+1}$ so:
    \begin{eqnarray*}
    f(p)f(q) &\geq &\beta\cdot 2^{\frac{p+q-d_{k-1}n_{k-1}}{2d_1\cdots d_{k-1}}+1-2^{-k}} \\
    &\geq & \beta\cdot 2^{\frac{n_k-d_{k-1}n_{k-1}}{2d_1\cdots d_{k-1}}+1-2^{-k}} \\
    &\geq &\beta\cdot 2^{\frac{n_k-d_{k-1}n_{k-1}}{2d_1\cdots d_{k-1}}+\frac{1}{2}} \\
    &\geq & f(p+q).
    \end{eqnarray*}

    \item Now if $p\leq d_{k-1}n_{k-1}$ (no restrictions on $q$, except that $q\leq p+q\leq d_kn_k$), then, taking $n_k>2d_{k-1}n_{k-1}$, it follows that $q>d_{k-1}n_{k-1}$. Now by $(\#)$:
    $$\frac{f(p+q)}{f(n_k)}\leq 2^{\frac{p+q-n_k}{2d_1\cdots d_k}} \leq 2^{\frac{p+q-n_k}{2d_1\cdots d_{k-1}}}$$ and by the definition of $f$:
    $$\frac{f(n_k)}{f(q)} \leq 2^{\frac{n_k-q}{2d_1\cdots d_{k-1}}}$$ so:
    $$\frac{f(p+q)}{f(q)} \leq 2^{\frac{p}{2d_1\cdots d_{k-1}}}\leq f(p)$$
    where the last inequality follows from Lemma \ref{lower bound f} for $x=p$ (keeping in mind that $p\leq d_{k-1}n_{k-1}$).
    \end{itemize}
    \item \textbf{It remains to take care of the case $d_{k-1}n_{k-1}\leq p\leq n_k\leq q$.} But notice that: $$f(p+q)\leq f(d_kn_k)\leq f(n_k)+d_k^2n_k^2\leq 2f(n_k)$$ since $f$ grows exponentially in the interval $[d_{k-1}n_{k-1}+1,n_k]$, so in particular we can take $n_k$ to be large enough such that $f(n_k)\gg d_k^2n_k^2$ (note that $d_k$ is already fixed when we choose $n_k$). Now: $$f(p)f(q)\geq 2f(q)\geq 2f(n_k).$$
    \end{itemize}
We thus proved that $f\colon\mathbb{N}\rightarrow \mathbb{N}$ is a submultiplicative function.
\end{proof}

\subsection{Our construction is not equivalent to any growth function}

Let $f\colon\mathbb{N}\rightarrow~\mathbb{N}$ be an increasing, submultiplicative function as constructed in Section \ref{construction of f}, with respect to sequences $\{d_k,n_k\}_{k=1}^\infty$.

\begin{prop} \label{notequiv}
We can choose $\{d_k,n_k\}_{k=1}^{\infty}$ such that the resulting function $f$ is not equivalent to the growth function of any finitely generated algebra.
\end{prop}
\begin{proof}
Since the conditions on $d_k$ and $n_k$ in Section~\ref{construction of f} are of the form:
$$ d_k\geq \mu_k \left(d_1,n_1,\dots,d_{k-1},n_{k-1}\right),\ \ \ d_k\geq \nu_k \left(d_1,n_1,\dots,d_{k-1},n_{k-1},d_k\right)$$
for suitable functions $\mu_k \colon \mathbb{N}^{2k-2}\rightarrow \mathbb{N},\ \nu_k \colon \mathbb{N}^{2k-1}\rightarrow \mathbb{N}$, we may further assume that $n_k=km_k$ for $m_k$ to be determined in the sequel.

Suppose that $C\in \mathbb{N}$ is given. Consider $k=C$, $n=m_k+1$ and $D=\lfloor d_k(1-~\frac{1}{m_k+1})\rfloor$, and observe that $\frac{1}{2}d_k\leq D\leq d_k$. If we take $d_k\geq 2k^2$ then it follows that $D\geq \frac{1}{2}d_k\geq k^2=C^2$. We aim to contradict the property stated in Proposition \ref{property_growth} with respect to these parameters, namely, we aim to show that: $$f(2CDn)-f(2CDn-C)>2D^2n\left(f(CDn)-f(Cn-C)\right)^{2D}.$$
Notice that $2CDn\leq 2kd_k(m_k+1)\leq 4d_kn_k$, and as long as we assume that $n_k>2k$, we have:
$$2CDn-C\geq 2CDn-CD\geq 2kDm_k\geq kd_km_k=d_kn_k$$
so, if we assume that $n_{k+1}>4d_kn_k$ then we may apply Condition (I) to $x=2CDn-C,y=2CDn$:
\begin{eqnarray*}
f(2CDn)-f(2CDn-C)& \geq & f(2CDn-C)\cdot 2^{\frac{C}{2d_1\cdots d_k}-2^{-k-2}}-f(2CDn-C) \\ & = &
f(2CDn-C)\cdot (2^{\frac{k}{2d_1\cdots d_k}-2^{-k-2}}-1) \\
& \geq & f(2CDn-C)\cdot \Delta_k
\end{eqnarray*}
where $\Delta_k=2^{\frac{k}{2d_1\cdots d_k}-2^{-k-2}}-1$ does not depend on $n_k$. 
Now, using Lemma \ref{lower bound f} for $x=d_kn_k$: $$f(2CDn-C)\geq f(d_kn_k)\geq 2^{\frac{d_kn_k}{2d_1\dots d_k}}=q_k^{n_k}=q_k^{km_k}$$ where $q_k=2^{\frac{1}{2d_1\cdots d_{k-1}}}$.
So, $f(2CDn)-f(2CDn-C)\geq q_k^{km_k}\cdot \Delta_k$.
On the other hand,
\begin{eqnarray*}
n_k = k m_k & = & Cn-C < CDn \\ & \leq & k\left(d_k\frac{m_k}{m_k+1}\right)(m_k+1)=d_kn_k
\end{eqnarray*}
so (by the definition of $f$ on intervals of this type): $$f(CDn)-f(Cn-C)\leq (CDn)^2\leq (kd_k(m_k+1))^2$$ hence:
\begin{eqnarray*}
2D^2n(f(CDn)-f(Cn-C))^{2D}& \leq & 2d_k^2(m_k+1)(kd_k(m_k+1))^{2d_k}\\
&\leq & (m_k+1)^{2d_k+1}\cdot \Gamma_k
\end{eqnarray*}
where $\Gamma_k=2d_k^2(kd_k)^{2d_k}$ depends only on $k,d_k$ (but not on $m_k$, or equivalently, on $n_k$).
Finally, note that as we fix $k,d_1,\dots,d_k$ and let $m_k\rightarrow \infty$ we get:
\begin{eqnarray*}
f(2CDn)-f(2CDn-C)&\geq &(q^k)^{m_k}\cdot \Delta_k \\
& \gg & (m_k+1)^{2d_k+1}\cdot \Gamma_k \\
&\geq & 2D^2n(f(CDn)-f(Cn-C))^{2D}
\end{eqnarray*}
contradicting the property ensured by Proposition \ref{property_growth} for functions which are equivalent to growth functions of algebras. Thus, $f$ is not equivalent to the growth of any finitely generated algebra.
\end{proof}

\subsection{Arbitrarily rapid holes: Proof of Theorem \ref{main thm submul}}

\begin{prop} \label{superpoly}
Let $g\colon\mathbb{N}\rightarrow \mathbb{N}$ be an arbitrary subexponential function. Then we can choose $\{d_k,n_k\}_{k=1}^{\infty}$ such that the resulting function in our construction $f$ satisfies $f\succeq g$.
\end{prop}
\begin{proof}

Since $g$ is assumed to be subexponential, there exists $\omega\colon\mathbb{N}\rightarrow \mathbb{R}$ such that $\omega(n)\xrightarrow{n\rightarrow \infty} 0$ and $g(n)\leq 2^{n\omega(n)}$.

We can take $n_k>\max\{m\ |\ \omega(m)\geq \frac{1}{2d_1\cdots d_k}\}$ for all $k\geq 1$.
We claim that for all $x\geq n_1$ we have that $f(x)\geq 2^{x\omega(x)}$. There are two possibilities: either $x\in [n_j,d_jn_j]$ or $x\in [d_jn_j,n_{j+1}]$ for some $j\geq 1$. Let us consider the first case. Then $x\geq n_j$ so by the way we picked $n_j$ we have that $\omega(x)<\frac{1}{2d_1\cdots d_j}$.
By Lemma \ref{lower bound f} we have that: $$f(x)\geq 2^{\frac{x}{2d_1\cdots d_j}}\geq 2^{x\omega(x)}.$$
As for the second case, if $x\in [d_jn_j,n_{j+1}]$ then by Condition (I), 
$$ f(x) \geq 2^{\frac{x-d_jn_j}{2d_1\cdots d_j}-2^{-j-2}}f(d_jn_j) $$
Now by Lemma \ref{lower bound f} for $d_jn_j$,
$$ f(d_jn_j) \geq 2^{\frac{d_jn_j}{2d_1\cdots d_j}+1+2^{-j-1}} $$
so:
\begin{eqnarray*}
f(x) & \geq & 2^{\frac{x-d_jn_j}{2d_1\cdots d_j}-2^{-j-2}} \cdot 2^{\frac{d_jn_j}{2d_1\cdots d_j}+1+2^{-j-1}} \\
& \geq & 2^{\frac{x}{2d_1\cdots d_j}}\\
& \geq & 2^{x\omega(x)}.
\end{eqnarray*}
Thus $f(x)\geq 2^{x\omega(x)}\geq g(x)$ for all $x\gg 1$.
\end{proof}


Finally, we have:

\begin{proof}[{Proof of Theorem \ref{main thm submul}}]
The theorem follows since we can take $\{d_k,n_k\}_{k=1}^\infty$ satisfying all conditions required in Propositions \ref{f submul}, \ref{notequiv} and \ref{superpoly}.
\end{proof}

\begin{rem} \label{Remark on exponentials}
If $f\colon\mathbb{N}\rightarrow\mathbb{N}$ is an increasing and submultiplicative function and $\limsup_{n\rightarrow\infty} \sqrt[n]{f(n)} > 1$ then $f\sim \exp(n)$, thus equivalent to the growth of (any) noncommutative free algebra.

Indeed, if $f(t_i)>a^{t_i}$ for some $t_1<t_2<\cdots$ with $t_i\leq 2^{s_i}\leq 2t_i$ then $f(2^{s_i})>~\sqrt{a}^{2^{s_i}}$; set $b=\sqrt{a}$. But if $f(2^s)>b^{2^s}$ then $f(2^{s-1})\geq \sqrt{f(2^s)}>b^{2^{s-1}}$, and since $s_1<s_2<\cdots$ are unbounded then for \textit{any} $t$, take $s$ such that $2^s\leq t< 2^{s+1}$ we have $f(t)\geq f(2^s)\geq b^{2^s}\geq \sqrt{b}^t$.

Thus, even though we did not explicitly prove it, it follows that $f$ from Theorem \ref{main thm submul} are always subexponential.
\end{rem}

\section{Growth rates encoding nilpotent ideals and non-prolongable words} \label{sec:encoding}

\subsection{Asymptotically invariant conditions on higher discrete derivatives}
We begin with a somewhat technical definition, which would be useful in the sequel:

\begin{defn}
We say that an increasing function $f\colon\mathbb{N}\rightarrow\mathbb{N}$ satisfies \textit{Condition (II)} if there exists $C\in \mathbb{N}$ such that for all $N,m\geq 1$: $$f'\left((C^2+1)m\right)\leq \left(f(2C^2N)-f(N)\right)\left(f\left(2C^2N+(C^4+C^2)m\right)-f(N+m)\right).$$
\end{defn}

We are now ready to translate certain algebraic properties of algebras to validity of Condition (II) for functions lying in the equivalence classes of their growth functions.

\begin{lem} \label{gen_gr}
Let $A=\bigoplus_{i=0}^{\infty} A_i$ be a graded algebra generated by $F+A_1\cdots+~A_p$. Then for any $r\in \mathbb{N}$ we have that $\bigoplus_{i=r}^{\infty} A_i$ is generated as a left ideal by $A_{r}+~\cdots +~A_{r+(p-1)}$.
\end{lem}
\begin{proof}
Let $L\leq A$ be the left ideal generated by $A_r+\cdots +A_{r+(p-1)}$. We prove by induction that every homogeneous element of degree $\geq r$ belongs to $L$. If $a\in A$ is homogeneous of degree $\deg(a) \in [r,r+p-1]$ the claim is evident. If $\deg(a)\geq r+p$, write $a=\sum f_{i_1}\cdots f_{i_k}$ where all $f_{i_j}$ are homogeneous of degrees $1~\leq~\deg(f_{i_j})~\leq~p$. Then each summand has the form $f_{i_1}w$ where $r\leq \deg(w)<~\deg(a)$, and so by the induction hypothesis $w\in L$ and hence each of these summands belongs to $L$, and consequently $a$ itself belongs to $L$ as well.
\end{proof}

\begin{prop} \label{preliminary_prime}
Let $f:\mathbb{N}\rightarrow \mathbb{N}$ be an increasing function which is equivalent to the growth of a finitely generated graded semiprime algebra. Then $f$ satisfies Condition (II).
\end{prop}

\begin{proof}
Let $A=\bigoplus_{i=0}^{\infty} A_i$ be a finitely generated semiprime algebra. 
Then it can be assumed that $A$ is generated by $V=F+A_1+\cdots+A_p$ for some $p\geq 1$ (this is true up to $A_0$, which is finite-dimensional). Let $\gamma$ be the growth function of $A$ with respect to $V$. Notice that $V^u=\bigoplus_{i=0}^{pu} A_i$ and consequently $\gamma(v)-\gamma(u)=\dim_F \bigoplus_{i=pu+1}^{pv} A_i$ for all $u<v$. Assume that $f\sim \gamma$, namely, there exists some $C\geq 1$ such that $f(n)\leq \gamma(Cn)$ and $\gamma(n)\leq f(Cn)$ for all $n\in \mathbb{N}$.

Fix $m,t,N,s\in \mathbb{N}$. Fix a basis $\mathfrak{B}$ of homogeneous elements for $\bigoplus_{i=pN+1}^{p(N+s)} A_i$.
Consider the linear map:
$$T \colon \bigoplus_{i=pm+1}^{p(m+t)} A_i \rightarrow \left(\bigoplus_{i=pN+1}^{p(N+s)} A_i\right) \otimes \left(\bigoplus_{i=p(m+N)+2}^{p(m+t+N+s)} A_i\right)$$
given by:
$$T \colon f\mapsto \sum_{u\in \mathfrak{B}} u\otimes uf$$
Then every element $f\in\Ker(T)$ satisfies $A_if=0$ for all $pN+1\leq i\leq pN+ps$ and therefore by Lemma \ref{gen_gr}, also $A_if=0$ for all $i\gg 1$. Hence $\left<f\right>\triangleleft A$ is nilpotent, so by semiprimeness necessarily $f=0$. Therefore $T$ is injective, and thus:
\begin{eqnarray*}
\gamma(m+t)-\gamma(m) & = & \dim_F \bigoplus_{i=pm+1}^{p(m+t)} A_i \\
& \leq & \left(\dim_F \bigoplus_{i=pN+1}^{p(N+s)} A_i\right) \cdot \left(\dim_F\bigoplus_{i=p(m+N)+2}^{p(m+t+N+s)} A_i\right) \\
& \leq & \left(\dim_F \bigoplus_{i=pN+1}^{p(N+s)} A_i\right) \cdot \left(\dim_F\bigoplus_{i=p(m+N)+1}^{p(m+t+N+s)} A_i\right) \\
& = & \left(\gamma(N+s)-\gamma(N)\right)\left(\gamma(N+m+t+s)-\gamma(N+m)\right)
\end{eqnarray*}
Scaling this by $C$ we get:
$$ \gamma(Cm+Ct)-\gamma(Cm) \leq \left(\gamma(CN+Cs)-\gamma(CN)\right)\left(\gamma(CN+Cm+Ct+Cs)-\gamma(CN+Cm)\right)$$
Now taking $t=C^2m,\ s=N$ and utilizing the connection between $f$ and $\gamma$, we deduce:
\begin{eqnarray*}
f'((C^2+1)m) & \leq & f((C^2+1)m)-f(C^2m) \\
& = & f(m+t)-f(C^2m)\\
& \leq & \gamma(Cm+Ct)-\gamma(Cm) \\
& \leq & \left(\gamma(CN+Cs)-\gamma(CN)\right)\left(\gamma(CN+Cm+Ct+Cs)-\gamma(CN+Cm)\right) \\
& \leq & \left(f(C^2N+C^2s)-f(N)\right)\left(f(C^2N+C^2m+C^2t+C^2s)-f(N+m)\right) \\
& = & \left(f(2C^2N)-f(N)\right)\left(f\left(2C^2N+(C^4+C^2)m\right)-f(N+m)\right)
\end{eqnarray*}
Thus $f$ satisfies Condition (II).
\end{proof}

We also have:

\begin{prop} \label{preliminary_discrete_derivative}
Let $f:\mathbb{N}\rightarrow \mathbb{N}$ be an increasing function which is equivalent to an increasing function $g:\mathbb{N}\rightarrow \mathbb{N}$ such that $g''(n)\geq 0$. Then $f$ satisfies Condition (II).
\end{prop}

\begin{proof}
Let $f,g$ be as in the statement of the proposition (in particular, $g'(n)\leq g'(n+t)$ for all $t\geq 0$). Then there exists $C \in \mathbb{N}$ such that $f(n)\leq g(Cn)$ and $g(n)\leq f(Cn)$ for all $n\in \mathbb{N}$. We also know that for every $n' \leq n$ and $t>0$:
$$g(n+t)-g(n) = \sum_{k=n+1}^{n+t} g'(k)\leq  \sum_{k=n+1}^{n+t} g'(k+n'-n)=\sum_{k=n'+1}^{n'+t} g'(k)=g(n'+t)-g(n')$$
Now taking $n=Cm,\ t=C^3m,\ n'=C(N+m)$ we obtain:
\begin{eqnarray*}
f'((C^2+1)m) & \leq & f((C^2+1)m)-f(C^2m) \\
& \leq & g((C^3+C)m)-g(Cm) \\
& = & g(n+t)-g(n) \\
& \leq & g(n'+t)-g(n') \\
& = & g(CN+Cm+C^3m) - g(CN+Cm) \\
& \leq & f(C^2N+C^2m+C^4m)-f(N+m) \\
& \leq & f\left(2C^2N+(C^4+C^2)m\right)-f(N+m)\\
& \leq & \left(f(2C^2N)-f(N)\right)\left(f\left(2C^2N+(C^4+C^2)m\right)-f(N+m)\right)
\end{eqnarray*}
Thus $f$ satisfies Condition (II).
\end{proof}

\subsection{Constructing an algebra}

We construct a function $f:\mathbb{N}\rightarrow \mathbb{N}$ inductively, then show it satisfies the conditions of \cite[Theorem~1.1]{BellZelmanov}, and deduce that it is equivalent to the growth function of a finitely generated monomial algebra.

\subsubsection{The growth function}
Suppose that a sequence of positive integers $n_1<~n_2<~\cdots$ is given, satisfying $n_{i+1}>n_i+i$ for all $i\in \mathbb{N}$ (along with further constraints to be made next).
Let $f(k)=2^k$ for all $k\leq 2^{n_1}$. For $k>2^{n_1}$ proceed as follows:
\begin{itemize}
    \item If $2^{n_i}<k\leq 2^{n_i+i}$ let $f(k)=f(k-1)+k+1$;
    \item If $2^{n_i+i}<k\leq 2^{n_{i+1}}$ let $f(k)=f(k-1)+\min\{f'(d)^2|k/2\leq d< k\}$.
\end{itemize}

By \cite[Theorem~1.1]{BellZelmanov}, in order to prove that $f$ is equivalent to the growth function of some finitely generated monomial algebra, it suffices to show:
\begin{enumerate}
    \item $f'(k)\geq k+1$ for all $k\geq 1$;
    \item $f'(k)\leq f'(n)^2$ for all $n\leq k\leq 2n$.
\end{enumerate}

It is evident that $f$ is increasing. We first show by induction that $f'(k)\geq k+1$ for all $k>1$. For $k=1,2,\dots,2^{n_1}$ this is clear (since $f(k)=2^k$ in this segment). Suppose that $k>2^{n_1}$ is given.
If $2^{n_i}<k\leq 2^{n_i+i}$ for some $i$ then this evidently follows from the definition of $f$, and otherwise, $f'(k)=f'(d)^2$ for some $\frac{k}{2} \leq d < k$ so by the induction hypothesis:
$$f'(k)=f'(d)^2\geq (d+1)^2\geq \left(\frac{k}{2}+1\right)^2\geq k+1.$$
We now show that $f'(k)\leq f'(n)^2$ for all $k\in \{n,n+1,\dots,2n\}$.
For $k=1,2,\dots,2^{n_1}$ this is clear since $f(k)=2^k$ in this segment.
If $2^{n_i+i}<k\leq 2^{n_{i+1}}$ for some $i$ then this is evident from the definition of $f$. Otherwise, suppose that $2^{n_i}<k\leq~2^{n_i+i}$. If $k/2\leq n<k$ then: $$f'(n)^2\geq (n+1)^2\geq \left(k/2+1\right)^2\geq k+1=f'(k).$$
It follows from \cite[Theorem~1.1]{BellZelmanov} that there exists a finitely generated graded algebra such that $\gamma_A(n)\sim f(n)$.

\subsubsection{Super-polynomial segments}

We now make some assumptions on $\{n_k\}_{k=1}^{\infty}$: we assume that for each $i\in \mathbb{N}$ we have $n_{i+1}>2(n_i+i+1)$ and $n_1 > 2$ (so $n_i > 2^i$), and derive a technical property of $f$ which will be utilized in the proofs of Theorems \ref{Prolongable} and \ref{Graded}.

\begin{prop} \label{aux}
For all $2^{n_{i+1}-1}<k\leq 2^{n_{i+1}}$ we have:
$$f'(k)\geq 2^{2^{\frac{1}{2}n_{i+1}}}.$$
\end{prop}

\begin{proof}
Observe that if $2^{n_i}<k\leq 2^{n_{i+1}}$ then $f'(k)\geq k+1> 2^{n_i}$. If in addition $k>2^{n_i+i}$ then $f'(k)=f'(d)^2$ for some $\frac{k}{2}\leq d<k$. 
In particular, $f'(d)\geq d+1>2^{n_i}$ so $f'(k)\geq 2^{2n_i}$.
By induction, it follows that for every $t>0$ such that: $$2^{n_i+i+t}<k\leq 2^{n_i+i+t+1}\leq 2^{n_{i+1}}$$
we have that:
$$f'(k)\geq 2^{2^tn_i}$$
Taking $t=n_{i+1}-n_i-i-1>\frac{1}{2}n_{i+1}$ we get that for $2^{n_{i+1}-1}<k\leq 2^{n_{i+1}}$: $$f'(k)\geq 2^{2^t n_i}\geq 2^{2^{\frac{1}{2}n_{i+1}}},$$
as claimed.
\end{proof}

\subsection{Non-primeness and non-prolongability}

We are now ready to prove Theorems \ref{Prolongable} and \ref{Graded}.
\begin{proof}[{Proof of Theorems \ref{Prolongable} and \ref{Graded}}]
Let $f$ be constructed as above, and $A$ a monomial algebra from \cite{BellZelmanov} whose growth function is equivalent to $f$. 
\begin{itemize}
    \item If $f\sim~\gamma$ for $\gamma \colon \mathbb{N}\rightarrow \mathbb{N}$ being the growth function of a graded, finitely generated semiprime algebra. Then by Proposition \ref{preliminary_prime} it is guaranteed that $f$ satisfies Condition (II).
    \item If $f$ is equivalent to the growth of a finitely generated graded algebra with a homogeneous regular element then $f\sim f_0$ for $f_0$ increasing with a non-decreasing derivative (as in Subsubsection \ref{subsubsec:graded algebras}, by transferring to a suitable Veronese subring).
Thus by Proposition \ref{preliminary_discrete_derivative}, $f$ satisfies Condition (II).
    \item If $f$ is equivalent to the growth of a prolongable hereditary language then it is equivalent to the growth of a graded algebra with a regular homogeneous element, see Subsection \ref{subsec:Hierarchy}, so again $f$ satisfies Condition (II).
\end{itemize}


We thus assume on the contrary that $f$ satisfies Condition (II), aiming for a contradiction.
In particular, for $m=\lfloor \frac{2^{n_{i+1}}}{C^2+1} \rfloor$ and $N=2^{n_{i+1}}$ we get:

\begin{eqnarray*} \label{equation1}
    f'\left((C^2+1)\left\lfloor \frac{2^{n_{i+1}}}{C^2+1} \right\rfloor\right) & \leq & 
    \left(f(2C^22^{n_{i+1}})-f(2^{n_{i+1}})\right) \cdot \\
    & &\cdot \left(f\left(2C^22^{n_{i+1}}+(C^4+C^2)\left\lfloor \frac{2^{n_{i+1}}}{C^2+1} \right\rfloor\right)-f\left(2^{n_{i+1}}+\left\lfloor \frac{2^{n_{i+1}}}{C^2+1} \right\rfloor\right)\right)
\end{eqnarray*}

We now make two observations:

\begin{itemize}
    \item Notice that the left hand side is $f'(k)$ for some $2^{n_{i+1}-1}<k\leq 2^{n_{i+1}}$, hence by Proposition \ref{aux}, assuming $i\gg_C 1$, it is bounded from below by:
    $$2^{2^{\frac{1}{2}n_{i+1}}}$$
    \item 
Notice that the right hand side is bounded from above by:
\begin{eqnarray*}
\left(f(2C^22^{n_{i+1}})-f(2^{n_{i+1}})\right)\left(f(3C^22^{n_{i+1}})-f(2^{n_{i+1}})\right) & \leq & \left(f(3C^22^{n_{i+1}})-f(2^{n_{i+1}})\right)^2 \\
& \leq & \left(\sum_{u=2^{n_{i+1}}+1}^{3C^22^{n_{i+1}+2}} f'(u)\right)^2
\end{eqnarray*}

For $i\gg_C 1$ we get that $3C^22^{n_{i+1}}<2^{n_{i+1}+i+1}$ and hence by construction, the above sum is bounded above by:
$$\left(\sum_{u=2^{n_{i+1}}+1}^{3C^22^{n_{i+1}}} u+1\right)^2$$
which is bounded from above by:
$$ (3C^2 2^{n_{i+1}}+1)^4 \leq C'2^{4n_{i+1}} $$
for suitable $C' \in \mathbb{N}$.
\end{itemize}
Thus, we obtain:
$$2^{2^{\frac{1}{2}n_{i+1}}}\leq C'2^{4(n_{i+1})}$$
which is false for $i\gg 1$, contradicting the assumption that $f$ satisfies Condition~(II). Theorems \ref{Prolongable} and \ref{Graded} now follow.
\end{proof}

Let $A$ be a monomial algebra. Let us say that a graded algebra $\widetilde{A}$ is a \textit{deformation} of $A$ if its Gr\"obner basis with respect to some monomial ordering coincides with that of $A$. Since $\gamma_A(n)=\gamma_{\widetilde{A}}(n)$, we immediately obtain from Theorem \ref{Graded} the following:

\begin{cor}
There exist finitely generated, infinite-dimensional monomial algebras all of whose graded deformations contain non-zero nilpotent ideals.
\end{cor}

\end{document}